\documentclass[letterpaper,11pt,reqno,tbtags]{amsart} 
\usepackage[portrait,margin=3cm]{geometry}  
\usepackage{mathrsfs}
\usepackage[
    colorlinks=true,%
    breaklinks,
    hyperindex,%
    plainpages=false,%
    bookmarksopen=false,%
    linkcolor=blue,%
    anchorcolor=blue,%
    citecolor=blue,%
    filecolor=black,%
    menucolor=black,%
    pagecolor=black,%
    urlcolor=blue,%
    pdfview=FitH,
    pdfstartview=FitH,
    hyperfigures,% backref=true,%pagebackref=true,%    
    bookmarksnumbered%
  ]{hyperref} 
\usepackage{amsmath,amssymb,amsthm,amsfonts,amsbsy,latexsym,dsfont,multirow,tikz}
\usepackage[center]{subfigure}

\newcommand{\vx}[0]{\mathbf x}

\newcommand{\vz}[0]{\mathbf z}
\newcommand{\vH}[0]{\mathbf H}
\newcommand{\vX}[0]{\mathbf X}

\newcommand{\mbf}[1]{\mbox{\boldmath$#1$}}

\newcommand{\calA}[0]{\mathcal{A}}
\newcommand{\calC}[0]{\mathcal{C}}
\newcommand{\calD}[0]{\mathcal{D}}
\newcommand{\calE}[0]{\mathcal{E}}
\newcommand{\calF}[0]{\mathcal{F}}
\newcommand{\calJ}[0]{\mathcal{J}}
\newcommand{\calN}[0]{\mathcal{N}}
\newcommand{\calS}[0]{\mathcal{S}}
\newcommand{\calT}[0]{\mathcal{T}}
\newcommand{\calV}[0]{\mathcal{V}}
\newcommand{\D}[0]{\mathbb{D}}
\newcommand{\E}[0]{\mathbb{E}}

\newcommand{\Id}[0]{\text{Id}}
\newcommand{\Proj}[0]{\mathcal{P}}
\def\varp{\boldsymbol\varepsilon}

\def\P{\mathbb{P}}

\newcommand{\Prob}[0]{\mathbb{P}}

\newcommand{\ind}[1]{\mathbb{I}\left\{#1\right\}}
\newcommand{\rbr}[1]{\left(#1\right)}
\newcommand{\rbR}[1]{\left[#1\right]}
\newcommand{\rBr}[1]{\left\{#1\right\}}
\newcommand{\rBR}[1]{\left|#1\right|}
\newcommand{\RBR}[1]{\left\|#1\right\|}

\newcommand{\bs}[1]{\boldsymbol{#1}}
\theoremstyle{remark}
\newtheorem{thm}{Theorem}[section]
\newtheorem{lem}[thm]{Lemma}

\newtheorem{prop}[thm]{Proposition}
\newtheorem{defn}[thm]{Definition}
\newtheorem{rem}[thm]{Remark}

\newtheorem{ass}[thm]{Assumption}

\newtheorem{ex}[thm]{Example}
   
\begin{document}

\title[Dynamic networks for time series]{Estimation of dynamic networks for high-dimensional nonstationary time series}
\author[Mengyu Xu]{Mengyu Xu}
\address{\newline Department of Statistics and Data Science\newline 
University of Central Florida \newline
4000 Central Florida Blvd, Orlando, FL 32816. \newline
{\it E-mail}: \href{mailto:Mengyu.Xu@ucf.edu}{\tt Mengyu.Xu@ucf.edu}
}

\author[Xiaohui Chen]{Xiaohui Chen}
\address{\newline Department of Statistics\newline
University of Illinois at Urbana-Champaign\newline
S. Wright Street, Champaign, IL 61820\newline
{\it E-mail}: \href{mailto:xhchen@illinois.edu}{\tt xhchen@illinois.edu}\newline 
{\it URL}: \href{http://publish.illinois.edu/xiaohuichen/}{\tt http://publish.illinois.edu/xiaohuichen/}
}

\author[Wei Biao Wu]{Wei Biao Wu}
\address{\newline Department of Statistics\newline
University of Chicago\newline
5747 S. Ellis Avenue, Jones 311, Chicago, IL 60637\newline
{\it E-mail}: \href{mailto:wbwu@galton.uchicago.edu}{\tt wbwu@galton.uchicago.edu}
}

\date{This version: \today}
%\subjclass[2010]{Primary: ; Secondary: ;}
\thanks{X. Chen's research is supported in part by NSF CAREER Award DMS-1752614 and UIUC Research Board Award RB18099. X. Chen acknowledges that part of this work is carried out at the MIT Institute for Data, System, and Society (IDSS). W.B. Wu's research is supported in part by NSF DMS-1405410.}

\begin{abstract}
This paper is concerned with the estimation of time-varying networks for high-dimensional nonstationary time series. Two types of dynamic behaviors are considered: structural breaks (i.e., abrupt change points) and smooth changes. To simultaneously handle these two types of time-varying features, a two-step approach is proposed: multiple change point locations are first identified on the basis of comparing the difference between the localized averages on sample covariance matrices, and then graph supports are recovered on the basis of a kernelized time-varying constrained $L_1$-minimization for inverse matrix estimation (CLIME)
 estimator on each segment. We derive the rates of convergence for estimating the change points and precision matrices under mild moment and dependence conditions. In particular, we show that this two-step approach is consistent in estimating the change points and the piecewise smooth precision matrix function, under a certain high-dimensional scaling limit.  The method is applied to the analysis of network structure of the S\&P 500 index between 2003 and 2008.
\end{abstract}
\maketitle

%\textbf{Keyword:} {High-dimensional time series, nonstationarity, network estimation, change points, kernel estimation.}

\section{Introduction}\label{sec:introduction_network}
%\textcolor{red}{Explain GGM and the partial correlation graphs}
%\footnote{This version: \today}
Networks are useful tools to visualize the relational information among a large number of variables. Undirected graphical model is a rich class of statistical network model that encodes the conditional independence~\cite{lauritzen1996graphical}. Canonically, Gaussian graphical models (or its normalized version partial correlation~\cite{pengwangzhouzhu2009a}) can be represented by the inverse covariance matrix (i.e., the precision matrix), where a zero entry is associated with a missing edge between two vertices in the graph. Specifically, two vertices are not connected if and only if they are conditionally independent given the value of all other variables.

On one hand, there is a large volume of literature on estimating the (static) precision matrix for graphical models in the high-dimensional setting, where the sample size and the dimension are both large~\cite{MR2278363,friedmanhastietibshirani2008a,MR2417243,rothmanbickellevinazhu2008a,yuan2010a,yuanlin2007a,ravikumarwainwrightraskuttiyu2008a,MR2382651, MR2847973, MR2847949,fanfengwu2009a,MR3335801,loh2014high,loh2013structure}. Most of the earlier work along this line assumes that the underlying network is time-invariant. This assumption is quite restrictive in practice and hardly plausible for many real-world applications such as gene regulatory networks, social networks, and stocking market, where the underlying data generating mechanisms are often dynamic. On the other hand, dynamic random networks have been extensively studied from the perspective of large random graphs such as community detection and edge probability estimation for dynamic stochastic block models (DSBMs) ~\cite{lebreeatal2010a,przytycka2010a,RePEc:eee:finmar:v:14:y:2011:i:1:p:1-46,chi2010network,durante2017nonparametric,durante2016locally,han2015consistent,danaher2014joint,dondelinger2013non,pensky2019dynamic,pensky2019spectral,bhattacharjee2018change,bartlett2018two,gaucher2019maximum}. Such approaches do not model the sampling distributions of the error (or noise), since the ``true" networks are connected with random edges sampled from certain probability models such as the Erd\H{o}s-R\'enyi graphs~\cite{ErdosRenyi59a} and random geometric graphs~\cite{penrose2003}.

In this paper, we shall view the (time-varying) networks of interests as non-random graphs. We adopt the graph signal processing approach for denoising the nonstationary time series and target on estimating the {\it true unknown} underlying graphs. Despite the recent attempts towards more flexible time-varying models~\cite{zhoulaffertywasserman2010a,kolarxing2011a,kolarsongxing2010a,kolarxing2014a,qiu2015joint,lu2015post,ahmedxing2009,tibshiranietal2005}, there are still a number of major limitations in the current high-dimensional literature. First, theoretical analysis was derived under the fundamental assumption that the observations are either temporally {\it independent}, or the temporal dependence has very specific forms such as Gaussian processes or (linear) vector autoregression (VAR)~\cite{zhoulaffertywasserman2010a,kolarxing2011a, MR3335801,cho2015multiple,royatchadmichailidis2014a, qiu2015joint,zhou2014gemini}. Such dynamic structures are unduly demanding in view that many time series encountered in real applications have very complex nonlinear spatial-temporal dependency~\cite{howell1993,fanyao2003}. Second, most existing work assumes the data have time-varying distributions with sufficiently light tails such as Gaussian graphical models and Ising models~\cite{zhoulaffertywasserman2010a,kolarxing2011a,cho2015multiple,royatchadmichailidis2014a,kolarxing2014a}. Third, in change point estimation problems for high-dimensional time series, piecewise constancy is widely used~\cite{royatchadmichailidis2014a,cho2015multiple,fryzlewicz2014wild,kokszkaleipus2000}, which can be fragile in practice. For instance, financial data often appears to have time-dependent cross-volatility with structural breaks~\cite{MR2572452}. For resting-state fMRI signals, correlation analysis reveals both slowly varying and abrupt changing characteristics corresponding to modularities in brain functional networks \cite{changglover2010a,hutchison2013a}.

Advances in analyzing high-dimensional (stationary) time series have been made recently to address the aforementioned the nonlinear spatial-temporal dependency issue~\cite{qiu2015joint, wiesel2013time, qiu2014robust, MR3335801,zhou2014gemini,chen2013covariance,ChenXuWu2016_IEEETSP,basu2015,MR3205718,shu2014estimation}. In \cite{chen2013covariance,MR3205718,shu2014estimation}, the authors considered the theoretical properties of regularized estimation of covariance and precision matrices, based on various dependence measure of high-dimensional time series. \cite{lu2015post} considered the non-paranormal graphs that evolves with a random variable.  \cite{qiu2015joint} discussed the joint estimation of Gaussian graphical models based on a stationary VAR(1) model with special coefficient matrices, which may also depend on certain covariates. The authors applied a constrained $L_1$-minimization for inverse matrix estimation (CLIME) estimator with a kernel estimator of covariance matrix and developed consistency in the graph recovery at a given time points. \cite{MR3335801} studied the recovery of the Granger causality across time and nodes assuming a stationary Gaussian VAR model with unknown order. %They applied a bi-level threshold group lasso-type estimator and discussed support recovery assuming restricted eigenvalues and irrepresentative condition. 

In this paper, we focus on the recovery of time-varying undirected graphs based on regularized estimation of the precision matrices for a general class of nonstationary time series. We simultaneously model two types of dynamics: abrupt changes with an unknown number of change points and the smooth evolution between the change points. In particular, we study a class of high-dimensional {\it piecewise locally stationary processes} in a general nonlinear temporal dependency framework, where the observation are allowed to have a finite polynomial moment.

More specifically, there are two main goals of this paper: first to estimate the change point locations, as well as the number of change points, and second to estimate the smooth precision matrix functions between the change points. Accordingly, our proposed method contains two steps. In the first step, the maximum norm of the local difference matrix is computed at each time point and the jumps in the covariance matrices are detected at the location where the maximum norms are above a certain threshold. In the second step, the precision matrices before and after the jump are estimated by a regularized kernel smoothing estimator. These two steps are recursively performed until a stopping criterion is met. Moreover, a boundary correction procedure based on data reflection is considered to reduce the bias near the change point.

We provide an asymptotic theory to justify the proposed method in high dimensions: point-wise and uniform rates of convergence are derived for the change point estimation and graph recovery under mild and interpretable conditions. The convergence rates are determined via subtle interplay among the sample size, dimensionality,  temporal dependence, moment condition, and the choice of bandwidth in the kernel estimator. Our results are significantly more involved than problems for sub-Gaussian tails and independent samples. We shall highlight that uniform consistency in terms of time-varying network structure recovery is much more challenging and difficult than pointwise consistency. %Stronger conditions on the networks regarding the synchronization are proposed, under which support consistency and related error control can be established. 
For the multiple change point detection problem, we also characterize the threshold of the difference statistic that gives consistent selection of the number of change points.

We fix some notation. Positive, finite and non-random constants, independent of the sample size $n$ and dimension $p$, are denoted by $C, C_1, C_2, \dots$, whose values may differ from line to line. For the sequence of real numbers, $a_n$ and $b_n$, we write $a_n=O(b_n)$ or $a_n\lesssim b_n$ if $\lim\sup_{n\to\infty} (a_n/b_n)\le C$ for some constant $C<\infty$ and $a_n=o(b_n)$ if $\lim_{n\to\infty}(a_n/b_n)=0$. We say $a_n\asymp b_n$ if $a_n=O(b_n)$ and $b_n=O(a_n)$.  For a sequence of random variables $Y_n$ and a corresponding set of constants $a_n$, denote $Y_n=O_\Prob(a_n)$ if for any $\varepsilon>0$ there is a constant $C>0$ such that $\Prob(|Y_n|/a_n> C)< \varepsilon$ for all $n$. For a vector $\vx \in \mathbb{R}^p$, we write $|\vx|= (\sum_{j=1}^p x_j^2)^{1/2}$. For a matrix $\Sigma$, $|\Sigma|_1 = \sum_{j,k}|\sigma_{j k}|$, $|\Sigma|_\infty = \max_{j,k} |\sigma_{jk}|$, $|\Sigma|_{L_1}=\max_k\sum_j |\sigma_{jk}|$, $|\Sigma|_F = (\sum_{j, k}\sigma_{j k}^2)^{1/2}$ and $\rho(\Sigma) = \max \{|\Sigma \vx| : |\vx|=1\}$. For a random vector $\vz \in \mathbb{R}^p$, write $\vz \in {\mathcal L}^a$, $a > 0$, if $\| \vz \|_a =: [ \E(|\vz|^a) ]^{1/a} < \infty$. Let $\RBR{\vz}=\RBR{\vz}_2$. Denote $a\wedge b = \min(a,b)$ and $a \vee b = \max(a,b)$. 

The rest of the paper is organized as following. Section~\ref{sec:model+assumptions} presents the time series model, as well as the main assumptions, which can simultaneously capture the smooth and abrupt changes. In Section~\ref{sec:method_network}, we introduce the two-step method that first segments the time series based on the difference between the localized averages on sample covariance matrices and then recovers the graph support based on a kernelized CLIME estimator. In Section~\ref{sec:theory_network}, we state the main theoretical results for the change point estimation and support recovery. Simulation examples are presented in Section~\ref{sec:simulation_network} and a real data application is given in Section~\ref{sec:real-data-analysis_network}. Proof of main results can be found in Section~\ref{sec:proof}.

\section{Time series model}
\label{sec:model+assumptions}

We first introduce a class of causal vector stochastic process. Then we state the assumptions to derive an asymptotic theory in Section~\ref{sec:theory_network} and explain their implications. Let $\varp_i\in \mathbb{R}^p, {i\in \mathbb{Z}}$ be independent and identically distributed (i.i.d.) random vectors and ${\calF}_i=(\ldots, \varp_{i-1},\varp_i)$ be a shift process. Let $\vX^\circ_i(t) = (X^\circ_{i1}(t), \dots, X^\circ_{ip}(t))$ be a $p$-dimensional nonstationary time series generated by
\begin{align}
\label{eqn:data_generetaion_mechnism}
\vX_i ^\circ(t)= \vH(\calF_i; \, t),
\end{align}
 where $\vH(\cdot;\cdot)=\big(H_1(\cdot;\cdot),\ldots,H_p(\cdot;\cdot))$ is an $\mathbb{R}^p$-valued jointly measurable function. %We rescale the time index to $t_i=i/n$ and %assume that the observations embed in the double-index sequence $(\vX_i^\circ(t_j))_{1\le i,j\le n}$, i.e., 
 Suppose we observe the data points $\vX_i=\vX_{i,n}=\vX^\circ_{i}(t_i)$ at the evenly spaced time intervals $t_i = i/n, i=1,2,\dots,n$,
\begin{align}\label{eqn:data_local}
\vX_{i,n}=\vH(\calF_i; \, i/n).
\end{align}
 We drop the subscription $n$ in $\vX_{i,n}$ in the rest of this section. Since our focus is to study the second-order properties, the data is assumed to be mean zero.

Model (\ref{eqn:data_generetaion_mechnism}) is first introduced in \cite{draghicescu2009quantile}. The stochastic process $\big(X^\circ_{i}(t)\big)_{i\in \mathbb Z, t\in[0,1)}$ can be thought as a triangular array system, double indexed by $i$ and $t$, while the observations $(X_i)_{i=1}^n$ are sampled from the diagonal of the array. On one hand, fixing the time index $t$, the (vertical) process  $\big(X^\circ_{i}(t)\big)_{i\in \mathbb Z}$ is stationary. On the other hand, since $\vH({\calF}_i;t_i)$ is allowed to vary with $t_i$, the diagonal process (\ref{eqn:data_local}) is able to capture nonstationarity.  

The process $(\vX_i)_{i\in \mathbb Z}$ is causal or non-anticipative as $\vX_i$ is an output of the past innovations $(\varp_{j})_{j\le i}$ and does not depend on the future innovations.  In fact, it covers a broad range of linear and nonlinear, stationary and non-stationary processes such as vector auto-regressive moving average processes, locally stationary processes, Markov chains, nonlinear functional processes~ \cite{MR2172215,draghicescu2009quantile,zhou2009local,zhouwu2010,chen2013covariance}.   

Motivated by real applications where nonstationary time series data can involve both abrupt breaks and smooth varies between the breaks, we model the underlying processes as piecewise locally stationary with a finite number of structural breaks.

\begin{defn}[Piecewise locally stationary time series model]
\label{defn:lp_function}
Define $\mathrm{PLS}_\iota([0, 1], L)$ as the collection of mean-zero piecewise locally stationary processes on $[0, 1]$, if for each $(X(t))_{0\le t\le 1} \in $ $\mathrm{PLS}_\iota([0, 1], L)$, there is a nonnegative integer $\iota$ such that $X(t)$ is piecewise stochastic Lipschitz continuous in $t$ with Lipschitz constant $L$ on the interval $[t^{(l)}, t^{(l+1)}), l = 0,\cdots,\iota$, where $0 = t^{(0)} < t^{(1)} \dots < t^{(\iota)} < t^{(\iota+1)} =1$. A vector stochastic process $(\vX(t))_{0\le t\le 1} \in \mathrm{PLS}_\iota([0, 1], L)$ if all coordinates belong to $\mathrm{PLS}_\iota([0, 1], L)$. For the process $(X^\circ_0(t))_{0\le t\le 1}$ defined in  (\ref{eqn:data_generetaion_mechnism}), this means that there exists a non-negative  integer $\iota$ and a constant $L>0$, such that
\[
\max_{1\le j \le p}\RBR{H_{j}(\calF_0;t)-H_{j}(\calF_0;t')}\le L|t-t'| \mbox{ for all } t^{(l)}\le t, t'< t^{(l+1)}, 0\le l\le \iota.
\]
\end{defn}

\begin{rem}
If we assume $(\vX_i^\circ(t))_{0\le t\le 1}\in \mathrm{PLS}_\iota([0, 1], L),i\in \mathbb Z$, then it follows that for each $i'=i-k,\ldots, i+k$, where $k/n\to 0$, and that $t^{(l)}\le i,i'<t^{(l+1)}$ for some $0\le l\le \iota$, we have
\[
\max_{1\le j\le p}\|H_{j}({\calF}_{i'};i/n)-H_{j}({\calF}_{i'};i'/n)\|\le Lk/n=o(1).
\]
In other words, within a locally stationary time period, in a local window of $i$, $(X_{i'j})_{i-k\le i'\le i+k}$ can be approximated by the stationary process $(X^\circ_{i'j}(i/n))_{i-k\le i'\le i+k}$  for each $j=1,\ldots, p$. This justifies the terminology of local stationarity.
\end{rem}

The covariance matrix function of the underlying process is $\Sigma(t)=\big(\sigma_{jk}(t)\big)_{1\le j,k\le p}$, $t\in [0,1]$, where $\sigma_{jk}(t)=\E\big(H_j(\calF_0;t)H_k(\calF_0;t))$, and the precision matrix function is $\Omega(t)= \Sigma(t)^{-1}= \big(\omega_{jk}(t)\big)_{1\le j,k\le p}$.  The graph at time $t$ is denoted by $G(t)=({\calV},{\calE}(t))$, where $\calV$ is the vertex set and ${\calE}(t)=\{(j,k): \omega_{jk}(t)\ne 0\}$. 
Note that $(\vX_i^\circ(t))_t \in \mathrm{PLS}_\iota([0, 1], L), i\in \mathbb Z$ implies piecewise Lipschitz continuity in $\Sigma(t)$ except at the breaks $t^{(1)}, \dots, t^{(\iota)}$. In particular, if $\sup_{0\le t\le 1} \max_{1\le j\le p}\RBR{H_j(\calF_0;t)}\le C$ for some constant $C>0$, then
\begin{eqnarray}
\label{eqn:L-cond-1-ture-covariance-matrix}
|\Sigma(s) - \Sigma(t)|_\infty &\le& 2CL |s-t|, \qquad \forall s,t \in [t^{(l)},t^{(l+1)}), l=0,\ldots, \iota.
\end{eqnarray}
The reverse direction is not necessarily true, i.e., (\ref{eqn:L-cond-1-ture-covariance-matrix}) does not indicate $(\vX_i^\circ(t))_t \in \mathrm{PLS}_\iota([0, 1], L)$, $i\in \mathbb Z$ in general. As a trivial example, let $\varepsilon_{ij}=2^{-1/2}$ with probability $2/3$ and $\sqrt 2$ with probability $1/3$ i.i.d for all $i,j$. At time $t_k=k/n$, let $X_{ij}^\circ(t_k)=(-1)^{k}\sqrt t_k\varepsilon_{ij}$. Then for any $k$ and $k'$ such that $k+k'$ is odd, $|\Sigma(t_k)-\Sigma(t_{k'})|_\infty=|t_k-t_{k'}|$, while $\|X^\circ_{01}(t_k)-X^\circ_{01}(t_{k'})\|_2=\sqrt{t_k}+\sqrt{t_{k'}}$.

\begin{ass}
[Piecewise smoothness]
\label{assumption:smoothness}
(i) Assume $(\vX_i^\circ(t))_{0\le t\le 1}\in \mathrm{PLS}_\iota([0, 1], L)$ for each $i\in \mathbb Z$, where $L>0$ and $\iota\ge 0$ are constants independent of $n$ and $p$. 

(ii) For each $l=0,\ldots,\iota$, and $1\le j,k\le p$, we have $\sigma_{jk}(t)\in{\calC}^2[t^{(l)},t^{(l+1)})$.
\end{ass}

%Assumptions \ref{assumption:smoothness} is standard in nonparametric estimation (\cite{kolarxing2014a,zhoulaffertywasserman2010a}). The smoothness condition on $\Sigma(\cdot)$ and $\Omega(\cdot)$ is not a restrictive assumption in the sense that smoothness of $\Sigma(\cdot)$ is implied by that of $\Omega(t)$ under the sparsity condition.

Now we introduce the temporal dependence measure. We quantify the dependence of $\big(\vX_i^\circ(t)\big)_{i\in \mathbb Z}$ by the dependence adjusted norm (DAN) (cf. \cite{wu2016}). Let $\varp'_{i}$ be an independent copy of $\varp_{i}$ and ${\calF}_{i,\{m\}}=(\dots,\varp_{i-m-1},\varp'_{i-m},\varp_{i-m+1},\dots,\varp_i)$. Denote $\vX^\circ_{i,\{m\}}(t)=\big(X^\circ_{i1,\{m\}}(t),\ldots, X^\circ_{ip,\{m\}}(t)\big)$, where $X^\circ_{ij,\{m\}}(t)=H_j({\calF}_{i,\{m\}}; t)$, $1\le j\le p$. Here $\vX^\circ_{i,\{m\}}(t)$ is a coupled version of $\vX^\circ_{i}(t)$, with the same generating mechanism and input, except that $\varp_{i-m}$ is replaced by an independent copy $\varp'_{i-m}$. 
\begin{defn}[Dependence adjusted norm (DAN)]
\label{defn:dan}
Let constants $a\ge1, A > 0$. Assume $\sup_{0\le t\le 1}\|X^\circ_{1j}(t)\|_a<\infty, j=1,\ldots,p$. Define the uniform functional dependence measure for the sequences $(X^\circ_{ij}(t))_{i\in \mathbb Z,t\in[0,1]}$ of form (\ref{eqn:data_generetaion_mechnism}) as
\[
\theta_{m,a,j}=\sup_{0\le t\le 1}\Vert X^\circ_{ij}(t)-X^\circ_{ij,\{m\}}(t)\Vert_a, \quad j = 1, \dots, p,
\]
and $\Theta_{m,a,j}=\sum_{i=m}^\infty\theta_{i,a,j}$. The dependence adjusted norm of $(X^\circ_{ij}(t))_{i\in \mathbb Z,t\in[0,1]}$ is defined as
\[
\RBR{X_{\cdot,j}}_{a,A}=\sup_{m \ge 0} (m+1)^A\Theta_{m,a,j},
\]
whenever $\RBR{X_{\cdot,j}}_{a,A} < \infty$.
\end{defn}

Intuitively, the physical dependence measure quantifies the adjusted stochastic difference between the random variable and its coupled version by replacing past innovations. Indeed, $\theta_{m,a,j}$ measures the impact on $X^\circ_{ij}(t)$ uniform over $t$ by replacing $\mbf\varepsilon_{i-m}$ while freezing all the other inputs, while  $\Theta_{m,a,j}$ quantifies the cumulative influence of replacing $\mbf\varepsilon_{-m}$ on $(X_{ij}^\circ(t))_{i\ge 0}$ uniform over $t$. Then $\RBR{X_{\cdot,j}}_{a,A}$ controls the uniform polynomial decay in the lag of the cumulative physical dependence, where $a$ depends on the  the tail of marginal distributions of $X^\circ_{1,j}(t)$ and $A$ quantifies the polynomial decay power and thus the temporal dependence strength. It is clear that $\RBR{X_{\cdot,j}}_{a,A}$ is a semi-norm, i.e., it is subaddative and absolutely homogeneous. 

\begin{ass}[Dependence and moment conditions]
\label{assumption:dep}
Let $\vX_i^\circ(t)$ be defined in (\ref{eqn:data_generetaion_mechnism}) and $\vX_i$ in (\ref{eqn:data_local}). There exist $q > 2$ and $A > 0$ such that
\begin{equation}
\label{eqn:moment_dan_vecproc}
\nu_{2q}:=\sup_{t\in [0,1]}\max_{1\le j\le p}\E|X^\circ_{j}(t)|^{2q} < \infty \qquad \text{and} \qquad N_{X,2q}:=\max_{1\le j\le p}\RBR{X_{\cdot, j}}_{2q,A} < \infty.
\end{equation}
\end{ass}
We let $M_{X,q}:=\rbr{\sum_{1\le j\le p}\RBR{X_{\cdot, j}}_{2q,A}^q}^{1/q}$ and write $N_X = N_{X,4}$, $M_X = M_{X,2}$. The quantities $M_{X,q}$ and $N_{X,2q}$ measure the $L^q$-norm aggregated effect and the largest effect of the element-wise DANs respectively.  Both quantities play a role in the convergence rates of our estimator. 

Obviously we have  $\|X_{ij}-X_{ij,\{m\}}\|_a\le \theta_{m,a,j}$ and $\max_{1\le j\le p}\E|X_{ij}|^{2q}\le \nu_{2q}$ for all $1\le i\le n$. 
In contrast to other works in high-dimensional covariance matrix and network estimation, where sub-gaussian tails and independence are the keys to ensure consistent estimation, Assumption \ref{assumption:dep} only requires that the time series have finite polynomial moment, and it allows linear and nonlinear processes with short memory in the time domain.

%\textcolor{red}{Is this assumption a necessary implication of PLS? or a stronger version of PLS? Neither.}
\begin{ex}[Vector linear process]\label{example:var}
Consider the following vector linear process model
\[
\vH(\calF_i;t)=\sum_{m=0}^{\infty} A_m(t)\mbf\varepsilon_{i-m},
\]
where $\mbf\varepsilon_i=(\varepsilon_1,\ldots,\varepsilon_p)$ and $\varepsilon_{ij}$ are i.i.d. with mean $0$ and variance $1$, and $\|\varepsilon_{ij}\|_{q}\le C_q$ for each $i\in \mathbb Z$ and $1\le j\le p$ with some constants $q>2$ and $C_q>0$.  The vector linear process is commonly seen in literature and application~\cite{Ltkepohl2007}. It includes the time-varying VAR model where $A_m(t)=A(t)^m$ as a special example.

Suppose that the coefficient matrices $A_m(t)=(a_{m,jk}(t))_{1\le j,k\le p}, m=0,1,\ldots$ satisfy the following condition.
\begin{enumerate}
\item[(A1)] For each $1\le j,k\le p$, $a_{m,jk}(t)\in\mathcal C^2[0,1].$
\item[(A2)] For each $1\le j\le p$, there is a constant $C_{A,j}>0$ such that for each $t\in [0,1]$, $\sum_{k=1}^p a_{m,jk}(t)^2\le C_{A,j}(m+1)^{-2(A+1)}$ for all $m\ge0$.
\item[(A3)] For any $t,t'\in[0,1]$, $\sum_{m=0}^\infty\sum_{k=1}^p[a_{m,jk}(t)-a_{m,jk}(t')]^2\le L^2|t-t'|^2$ for each $j=1,\ldots,p$.
\end{enumerate}

Note that 
\begin{align*}
\sigma_{jk}(t)&=\sum_{m\ge 0}A_{m,j\cdot}^\top(t) A_{m,k\cdot}(t),\\
\Theta_{m,q,j}&\le 2C_q\sqrt{q-1}\sum_{m=0}^\infty (A_{m,j\cdot}^\top A_{m,j\cdot})^{1/2},\\
\|X_{ij}^\circ(t)-X_{ij}^\circ(t')\|^2&=\sum_{m=0}^\infty A_{m,j\cdot}\sum_{k=1}^p[a_{m,jk}(t)-a_{m,jk}(t')]^2,
\end{align*}
 where $A_{m,j\cdot}(t)$ is the $j$th row of $A_{m}(t)$. Under condition (A1)-(A3), one can easily verify that for each $1\le j,k\le p$, the process satisfies: (1)  $\sigma_{jk}(t)\in{\calC}^2[0,1]$; (2) $\|X_{\cdot,j}\|_{q,A}\le C_q \sqrt{(q-1)C_{A,j}}$ (due to Burkholder's inequality, cf. \cite{MR1476912}); (3) $\|H_j(\calF_0;t)-H_j(\calF_0;t')\|\le L|t-t'|$.

Conditions (A1)-(A3) implicitly impose smoothness in each entry of the coefficient matrices, sparseness in each column of the entry and evolution, and polynomial decay rate in the lag $m$ of each entry and its derivative. 
\end{ex}

For $1\le l\le \iota$, let $\delta_{jk}(t^{(l)}):=\sigma_{jk}(t^{(l)})-\sigma_{jk}(t^{(l)}-)$ and $\Delta(t^{(l)})=\big(\delta_{jk}(t^{(l)})\big)_{1\le j,k\le p}$, where $\sigma_{jk}(t^{(l)}-)=\lim_{t\to t^{(l)}-}\sigma_{jk}(t)$ is well-defined in view of (\ref{eqn:L-cond-1-ture-covariance-matrix}). We assume that the change points are separated and sizeable.
\begin{ass}[Separability and sizeability of change points]
\label{assumption:jump}
There exist positive constants $c_1\in(0,1)$ and $c_2>0$ independent of $n$ and $p$ such that $\max_{0\le l\le \iota}(t^{(l+1)}-t^{(l)})\ge c_1$ and $\delta(t_l):= |\Delta(t_l)|_\infty\ge c_2$.
\end{ass}

In the high-dimensional context, we assume that the inverse covariance matrices are sparse in the sense of their $L_1$ norms. %In particular, we consider the following class of matrices

\begin{ass}[Sparsity of precision matrices]
\label{assumption:sparsity}
The precision matrix $|\Omega(t)|_{L^1}\le \kappa_p$ for each $t\in[0,1]$, where $\kappa_p$ is allowed to grow with $p$.
\end{ass}

If we further assume that the eigenvalues of the covariance matrices are bounded from below and above, i.e., there exists a constant $0<c<1$ such that $c\le \inf_{t\in [0,1]}|\Sigma(t)|_2\le \sup_{t\in [0,1]}|\Sigma(t)|_2\le c^{-1}$, then the covariance matrices and precision matrices are well-conditioned. In particular, as $|\Omega(t)-\Omega(t')|\le c^{-2}|\Sigma(t)-\Sigma(t')|$, a small perturbation in the covariance matrix would guarantee a small change of the same order in the precision matrix under the spectral norm.

%If $|\Sigma(s) - \Sigma(t)|_2 \le L |s-t|$, then $|\Omega(s) - \Omega(t)|_2\le c^{-2}L |s-t|$.

\section{Method: change point estimation and support recovery}
\label{sec:method_network}
In graphical models (such as Gaussian graphical model or partial correlation graph), network structures relevant to correlations or partial correlations are second-order characteristics of the data distributions. Specifically, existence of edges coincides with non-zero entries of the inverse covariance matrix. We consider the dynamics of time series with both structural breaks and smooth changes. The piecewise stochastic Lipschitz continuity in Definition~\ref{defn:lp_function} allows the time series to have discontinuity in the covariance matrix function at time points $t^{(l)}, l=1,\dots,\iota$ (i.e., change points), while only smooth changes (i.e., twice continuous differentiability of the covariance matrix function in Assumptions~\ref{assumption:smoothness}) can occur between the change points.

In the presence of change points, we must first remove the change points before applying any smoothing procedures since $|\Omega(t)-\Omega(t-)|_\infty\ge |\Sigma(t)|_{L^1}^{-1} |\Sigma(t-)|_{L^1}^{-1}|\Delta(t)|_\infty$, i.e., a non-negligible abrupt change in the covariance matrix will result in a substantial change of the graph structure for sparse and smooth covariance matrices. 
%More importantly, recovering the graph structure using time series data that has abrupt jumps in the covariance matrix will induce remarkable error, as the estimates of the covariance components are themselves inaccurate. In other words, local smoothness in the covariance matrix is required for the inference of the graphical structure via constraint $l_1$ minimization estimator.
Thus our proposed graph recovery method consists of two steps: change point detection and support recovery.

Let $h \equiv h_n>0$ be a bandwidth parameter such that $h=o(1)$ and $n^{-1}=o(h)$, and ${\calD}_h(0) = \{ h, h+1/n, \dots, 1-h\}$ be a search grid in $(0,1)$. Define
\begin{equation}
\label{eqn:D_nonstationary}
D(s) = n^{-1} \left( \sum_{i=0}^{hn-1} \vX_{ns-i} \vX_{ns-i}^\top - \sum_{i=1}^{hn} \vX_{ns+i} \vX_{ns+i}^\top \right), \qquad s \in {\calD}_h(0).
\end{equation}
To estimate the change points, compute
\begin{equation}
\label{eqn:change-point-estimator-infty-nonstationary1}
\hat{s}_1 = \text{argmax}_{s \in {\calD}_h(0)} |D(s)|_\infty.
\end{equation}
The following steps are performed recursively. For $l=1,2,\ldots$, let
\begin{eqnarray}
\label{eqn:change-point-estimator-infty-nonstationary}
&{\calD}_h(l) = {\calD}_h(l-1)\cap\{\hat s_l-2h,\cdots,\hat s_l+2h\}^c,\\
&\hat{s}_{l+1} = \arg\max_{s \in {\calD}_h(l)} |D(s)|_\infty,
\end{eqnarray}
until the following criterion is attained:
\begin{align}\label{eqn:early_stop}
\max_{s\in {\calD}_h(l)} |D(s)|_\infty < \nu,
\end{align}
where $\nu$ is an early stopping threshold. The value of $\nu$ is determined in Section \ref{sec:theory_network}, which depends on the dimension and sample size, as well as the serial dependence level, tail condition and local smoothness. Since our method only utilizes data in the localized neighborhood, multiple change points can be estimated and ranked in a single pass, which offers some computational advantage than the binary segmentation algorithm \cite{cho2015multiple,fryzlewicz2014wild}.

%Future work: The choice of $\nu$ is equivalent to the construction of simultaneous confidence band of $\delta(t)$, or the hypothesis test problem of
%\[
%H_0: \iota=0\qquad\mbox{ v.s. } H_A: \iota\ge 1.
%\] 

%\begin{rmk}
%Similar change-point estimator as (\ref{eqn:change-point-estimator-infty-nonstationary1}) using the maximum aggregation appeared in \cite{groenkapetaniosprice2013}. Our estimator is also related to the binary segmentation algorithm proposed in \cite{cho2015multiple}, where the piecewise constant models are studied and a thresholded $\ell^1$ aggregation is considered, where an additional thresholding parameter is introduced. Here, as a key difference, we use the maximum aggregation strategy, which is tuning-free and whose consistency is guaranteed by Theorem \ref{theorem:piecewise-nonstationary-rate-block-version}. \qed
%\end{rmk}

Once the change points are claimed, in the second step,we consider recovering the networks from the locally stationary time series before and after the structural breaks. In~\cite{MR2847973}, where $X_i,i=1,\ldots,n$ are assumed with an identical covariance matrix, the precision matrix $\hat \Omega$ is estimated as,
\begin{align}\label{eq:clime}
    \hat\Omega_\lambda=\arg\min_{\Omega\in \mathbb R^{p\times p}} |\Omega|_1\quad \mbox{s.t. }|\hat\Sigma\Omega-\Id_p|_\infty\le \lambda,
\end{align}
where $\hat\Sigma$ is the sample covariance matrix. Inspired by (\ref{eq:clime}), we apply a kernelized time-varying (tv-) CLIME estimator for the covariance matrix functions of the multiple pieces of locally stationary processes before and after the structural breaks. Let
\begin{align}
\label{eqn:samplecov}
\hat\Sigma(t)=\sum_{i=1}^{n} w(t,t_i) \vX_i \vX_i^\top,
\end{align}
where
\begin{align}\label{eqn:kernel weight}
w(t,i)=\frac{K_b(t_i,t)}{\sum_{i=1}^n K_b(t_i,t)}
\end{align}
and $K_b(u,v)=K(|u-v|/ b )/ b $. The bandwidth parameter $b$ satisfies that $b=o(1)$ and $n^{-1}=o(b)$. Denote $B_n=nb$. The kernel function $K(\cdot)$ is chosen to have properties as follows.
\begin{ass}[Regularity of kernel function]\label{assumption:kernel}
The kernel function $K(\cdot)$ is non-negative, symmetric, and Lipschitz continuous with bounded support in $[-1,1]$, and that $\int_{-1}^1K(u)du=1$.
\end{ass}

Assumption \ref{assumption:kernel} is a common requirement on the kernel functions and  can be fulfilled by a range of kernel functions such as the uniform kernel, triangular kernel, and the Epanechnikov kernel. Now the tv-CLIME estimator of the precision matrix $\Omega(t)$ is defined by $\tilde\Omega(t)=\left(\tilde{\omega}_{jk}(t)\right)_{1\le j,k\le p}$, where  
$\tilde{\omega}_{jk}(t)=\min(\hat\omega_{jk}(t),\hat\omega_{kj}(t))$, and $\hat \Omega(t)\equiv \hat\Omega_\lambda(t)=(\hat\omega_{jk}(t))_{1\le j,k\le p}$,
\begin{align}
\label{eqn:CLIME}
\hat\Omega_\lambda(t)=\arg\min_{\Omega\in \mathbb R^{p\times p}} |\Omega|_1\quad \mbox{s.t. }|\hat\Sigma(t)\Omega-\Id_p|_\infty\le \lambda.
\end{align} 
Similar hybridized kernel smoothing and CLIME method for estimating the sparse and smooth transition matrices in high-dimensional VAR model has been considered in~\cite{ding2017}, where change point is not considered. Thus in the current setting we need to carefully control effect of (consistently) removing the change points before smoothing.

Then, the network is estimated by the ``effective support" defined as follows.
\begin{align}
\label{eqn:graph_estimation}
\hat G(t;u)=(\hat g_{jk}(t;u))_{1\le j,k\le p}, \quad\mbox{where   } \, \hat g_{jk}(t;u)=\ind{|\tilde{\omega}_{jk}(t)| \ge u}.
\end{align}

 It should be noted that the (vanilla) kernel smoothing estimator~\eqref{eqn:samplecov} of the covariance matrix does not adjust for the boundary effect due to the change points in the covariance matrice function. Thus, in the neighborhood of the change points, larger bias can be induced in estimating $\Sigma(t)$ by $\hat \Sigma(t)$. 
As a remedy, we apply the following reflection procedure for boundary correction. Suppose $t\in\hat{\mathcal{T}}_{b+ h^2}(j)$ for $1\le j\le \iota$, Denote $\hat{{\calT}}_{d}(j):=[\hat s_j-d,\hat s_j+d)$ for $d\in (0,1)$. We replace (\ref{eqn:samplecov}) by
\[
\hat\Sigma(t)=\sum_{i=1}^{n} w(t,t_i) \breve\vx_i \breve\vx_i^\top ,
\]
and then apply the rest of the tv-CLIME approach. Here
\begin{align}
\breve\vx_i=\begin{cases}
\vx_i&\mbox{if }(i-\hat s_jn)(t-\hat s_jn)\ge 0;\\
\vx_{2\hat s_jn-i}&\mbox{otherwise}.
\end{cases}
\end{align}

\section{Theoretical results}
\label{sec:theory_network}

In this section, we derive the theoretical guarratees for the change point estimation and graph support recovery. Roughly speaking, Proposition \ref{prop:precmx_smooth} and \ref{prop:partial_support_recovery} below show that under appropriate conditions, if each element of the covariance matrix varies smoothly in time, one can obtain accurate snapshot estimation of the precision matrices as well as  the time-varying graphs with high probability via the proposed kernel smoothed constrained $l_1$ minimization approach.

%Our first result is that under appropriate conditions, element-wise local smoothness of the covariance matrix is sufficient for the consistency of support recovery via tv-CLIME. The latter observation motivates us  to first identify the abrupt changes in covariance matrices in the sense of the max norm.  

Define $J_{q,A}(n,p) = M_{X,q} (p  \varpi_{q,A}(n)) ^{1/q}$, where $\varpi_{q,A}(n) = n, n(\log n)^{1+2q}, n^{q/2-Aq}$ if $A > 1/2-1/q$, $A = 1/2-1/q$, and $0 < A < 1/2-1/q$, respectively.

\begin{prop}[Rate of convergence for estimating precision matrices: pointwise and uniform]\label{prop:precmx_smooth}
Suppose Assumptions \ref{assumption:dep}, \ref{assumption:sparsity} and \ref{assumption:kernel} hold with $\iota=0$. Let $B_n=bn$ for $n^{-1}=o(b)$ and $b=o(1)$.  

\begin{enumerate}
\item[(i)] {\bf Pointwise.} Choose the parameter $\lambda^\circ \ge C\kappa_p ( b ^2+B_n^{-1}J_{q,A}(B_n,p)+N_X (\log{p}/B_n)^{1/2})$ in the tv-CLIME estimator $\hat\Omega_{\lambda^\circ}(t)$ in (\ref{eqn:CLIME}),  where $C$ is a sufficiently large constant independent of $n$ and $p$. Then for any $t\in [b,1-b]$, we have
\begin{align}\label{eqn:pre_dev}
|\hat\Omega_{\lambda^\circ}(t)-\Omega(t)|_\infty&= O_\P( \kappa_p\lambda^\circ).
\end{align}

\item[(ii)] {\bf Uniform.} Choose $\lambda^\diamond \ge C \kappa_p \left( b ^2+B_n^{-1} J_{q,A}(n,p)+N_X B_n^{-1}(n\log(p))^{1/2}\right)$ in the tv-CLIME estimator $\hat\Omega_{\lambda^\circ}(t)$ in (\ref{eqn:CLIME}),  where $C$ is a sufficiently large constant independent of $n$ and $p$. Then we have 
\begin{align}
\label{eqn:pre_dev_unif}
\sup_{t \in [b,1-b]} |\hat\Omega_{\lambda^\diamond}(t)-\Omega(t)|_\infty = O_\P( \kappa_p\lambda^\diamond ).
\end{align}
\end{enumerate}
\end{prop}

The optimal order of the bandwidth parameter $b= b_\sharp$ in~\eqref{eqn:pre_dev_unif} is the solution to the following equation:
\begin{eqnarray*}
b^2 &=& B_n^{-1}\max( J_{q,A}(n,p), \, N_X(n\log(p^2))^{1/2}),
\end{eqnarray*}
which implies that the closed-form expression for $b_\sharp$ is given by
\begin{align*}
b_\sharp=C_1\big(n^{-1}J_{q,A}(n,p)\big)^{1/3}+C_2 N_X^{1/3}n^{-1/6}\log(p)^{1/6}
\end{align*}
for some constants $C_1$ and $C_2$ that are independent of $n$ and $p$.

Given a finite sample, to distinguish the small entries in the precision matrix from the noise is challenging. Since a smaller magnitude of a certain element of the precision matrix implies a weaker connection of the edge in the graphical model, we instead consider the estimation of {\it significant} edges in the graph. Define the set of {\it significant} edges at level $u$ as ${\calE}^*(t;u)=\rBr{(j,k):g^*_{jk}(t;u)\ne0}$, where 
\[
g^*_{jk}(t;u)=\ind{|\omega_{jk}(t)|>u}.
\]
Then, as a consequence of (\ref{eqn:pre_dev_unif}),  we have the following support recovery consistency result.

\begin{prop}[Consistency of support recovery: significant edges]
\label{prop:partial_support_recovery}
Choose $u$ as $u_\sharp = C_0 \kappa_p^2 b_\sharp^2$, where $C_0$ is taken as a sufficiently large constant independent of $n$ and $p$. Suppose that $u_\sharp=o(1)$ as $n,p\to\infty$. Then under conditions of Proposition \ref{prop:precmx_smooth}, we have that as $n,p\to\infty$,
\begin{align}\label{eqn:fpos}
\P\Big(\sup_{t \in [b,1-b]} \sum_{(j,k)\in {\calE}^c(t)}\ind{\hat g_{jk}(t;u_\sharp)\ne 0}\ne 0\Big)\to 0,\\\label{eqn:fneg}
\P\Big(\sup_{t \in [b,1-b]} \sum_{(j,k)\in {\calE}^*(t;2u_\sharp)}\ind{\hat g_{jk}(t;u_\sharp)=0}\ne 0\Big)\to0.
\end{align}
\end{prop}
Proposition \ref{prop:partial_support_recovery} shows that the pattern of significant edges in the time-varying true graphs $G(t), t\in[b,1-b]$, can be correctly recovered with high probability. However, it is still an open question to what extent the edges with magnitude below $u$ can be consistently estimated, which can be naturally studied in the multiple hypothesis testing framework. Nonetheless, hypothesis testing for graphical models on the nonstationary high-dimensional time series is rather challenging. We leave it as a future problem.

Propositions \ref{prop:precmx_smooth} and \ref{prop:partial_support_recovery} together yield that consistent estimation of the precision matrices and the graphs can be achieved before and after the change points.  Now we provide theoretical result of the change point estimation. Theorem \ref {thm:jumppointsestimation} below shows that if the change points are separated and sizeable, then we can consistently identify them via the single pass segmentation approach under suitable conditions. Denote
\[
h_\diamond=C_1\big(n^{-1}J_{q,A}(n,p)\big)^{1/3}+C_2 N_X^{1/3}n^{-1/6}\log(p)^{1/6},
\]
where $C_1$ and $C_2$ are constants independent of $n$ and $p$.

\begin{thm}[Consistency of change point estimation]
\label{thm:jumppointsestimation}
Assume $\vX_i\in \mathbb R^p$ admits the form (\ref{eqn:data_local}). Suppose that Assumptions \ref{assumption:dep} to \ref{assumption:jump} are satisfied. Choose the bandwidth $h= h_\diamond$, and $\nu=(1+L) h_\diamond^2$ in (\ref{eqn:D_nonstationary}) and (\ref{eqn:early_stop}) respectively. Assume that $h_\diamond=o(1)$ as $n,p\to\infty$. We have that there exist constants $C_1,C_2,C_3$ independent of $n$ and $p$ such that
\begin{eqnarray}\label{eqn:jumpnumber}
\P(|\hat{\iota}-\iota|> 0)\le C_1\Big({p \varpi_{q,A}(n) M_{X,q}^q \nu_{2q}^q \over n^{q} c_2^q}\Big)^{1/3}+C_2p^2\exp\Big\{-C_3({n\log^2(p)\over N_X^2})^{1/3}\Big\}.
\end{eqnarray}
Furthermore, on the event $\{\iota=\hat\iota\}$, the ordered change-point estimator $(\hat{s}_{(1)}<\hat{s}_{(2)}<\cdots<\hat{s}_{(\hat \iota)})$ defined in (\ref{eqn:change-point-estimator-infty-nonstationary}) satisfies
\begin{eqnarray}
\label{eqn:jumppoint}
&\max_{1\le j\le \iota}|\hat{s}_{(j)}-t^{(j)}| = O_\Prob(h^2_\diamond).
\end{eqnarray}
\end{thm}
Proposition \ref{prop:partial_support_recovery} and Theorem \ref{thm:jumppointsestimation} together indicate the consistency in the snapshot estimation of the time-varying graphs before and after the change points. In a close neighborhood of the change points, we have the following result for the recovery of the time-varying network. Denote ${\calS}:=\big[b_\sharp,1-b_\sharp]\cap(\cup_{1\le j\le \hat\iota}\hat{\calT}^c_{h_\diamond^2+b_\sharp}(j)\big)$ as the time intervals between the estimated change points, and ${\calN}:=[0,b_\sharp)\cup\big(\cup_{1\le j\le \hat\iota}(\hat{\calT}_{h_\diamond^2+b_\sharp}\cap\hat{\calT}^c_{h_\diamond^2})\big)\cup(1-b_\sharp,1]$ as the recoverable neighborhood of the jump.

\begin{thm}\label{thm:support_recovery}
Let Assumptions \ref{assumption:dep} to \ref{assumption:kernel} be satisfied.  We have the following results as $n,p\to \infty$.
\begin{enumerate}
\item[(i)] {\bf Between change points.} For $t\in{\calS}$, take $b=b_\sharp$ and $u=u_\sharp$, where $b_\sharp$ and $u_\sharp$ are defined in Proposition \ref{prop:partial_support_recovery}. Suppose $u_\sharp=o(1)$. we have
\begin{align}\label{eqn:cov_boundary_unif}
&\sup_{t\in{\calS}}\max_{j,k}|\hat\sigma_{j,k}(t)-\sigma_{j,k}(t)|= O_\P( b^2_\sharp).
\end{align}
Choose the penalty parameter as $\lambda_\sharp:=C_1\kappa_pb^2_\sharp$, where $C_1$ is a constant independent of $n$ and $p$. Then
\begin{align*}
%\label{pre_boundary_unif}
\sup_{t\in{\calS}}|\hat{\Omega}_{\lambda_\sharp}(t)-\Omega(t)|_\infty= O_\P( \kappa_p^2  b_\sharp ^2).
\end{align*}
Moreover, 
\begin{align}
\P\Big(\sup_{t\in{\calS}}\sum_{(j,k)\in {\calE}^c(t)}\ind{\hat g_{j,k}(t;u_\sharp)\ne 0}= 0\Big)\to 1,\\
\P\Big(\sup_{t\in{\calS}} \sum_{(j,k)\in {\calE}^*(t;2u_\sharp)}\ind{\hat g_{jk}(t;u_\sharp)=0}= 0\Big)\to 1.
\end{align}

\item[(ii)] {\bf Around change points.} For $s\in{\calN}$, take $b=b_\star:=C_1\big(n^{-1}J_{q,A}(n,p)\big)^{1/2}+C_2 N_X^{1/2}n^{-1/4}\log(p)^{1/4}$, and  $u=u_\star: = C_0 \kappa_p^2 b_\star$, where $C_0$, $C_1$ and $C_2$ are constants independent of $n$ and $p$. Suppose $u_\star=o(1)$. We have
\begin{align*}%\label{eqn:cov_boundary_unif}
&\sup_{t\in{\calN}}\max_{j,k}|\hat\sigma_{j,k}(t)-\sigma_{j,k}(t)|= O_\P( b_\star).
\end{align*}
Choose the penalty parameter as $\lambda_\star:=C_1\kappa_pb_\star$, where $C_1$ is a constant independent of $n$ and $p$. Then
\begin{align}
\label{pre_boundary_unif}
\sup_{t\in{\calN}}|\hat{\Omega}_{\lambda_\star}(t)-\Omega(t)|_\infty= O_\P( \kappa_p^2  b_\star).
\end{align}
Moreover,
\begin{align}
\P\Big(\sup_{t\in{\calN}}\sum_{(j,k)\in {\calE}^c(t)}\ind{\hat g_{j,k}(t;u_\star)\ne 0}= 0\Big)\to 1,\\
\P\Big(\sup_{t\in{\calN}} \sum_{(j,k)\in {\calE}^*(t;2u_\star)}\ind{\hat g_{j,k}(t;u_\star)=0}= 0\Big)\to 1.
\end{align}
\end{enumerate}
\end{thm} 

Note that the convergence rates for the covariance matrix entries and precision matrix entries in case (ii) around the jump locations are slower than those for points well separated from the jump locations in case (i). This is because on the boundary due to the reflection, the smooth condition may no longer holds true. Indeed, we only take advantage of the Lipschitz continuous property of the covariance matrix function. Thus we lose one degree of regularity in the covariance matrix function, and the bias term $b^2$ in the convergence rate of the between-jump area becomes $b$ around the jumps. 
We also note that around the smaller neighborhood of the jump ${\calJ}:=\cup_{1\le j\le \hat\iota}\hat{\calT}_{h_\diamond^2}$, due to the larger error in the change point estimation, consistent recovery the graphs is not achievable. 

%As a result, in performing the estimation, the optimal theoretical bandwidth parameter $b$ is smaller than that for the between jumps, and so is the penalty parameter $\lambda$.

\section{A simulation study}\label{sec:simulation_network}
We simulate data from the following multivariate time series model:
\[
X_i=\sum_{m=0}^{100}A_m(i)\bs\epsilon_{i-m}, i=1,\ldots,n,
\]
where $A_{m}(i)\in\mathbb R^{p\times p},1\le m\le 100,1\le i\le n$, and $\bs\epsilon_{i-m}=(\epsilon_{i-m,1},\ldots,\epsilon_{i-m,p})^\top$, with $\epsilon_{m,k}$, $m\in \mathbb Z$, $j=1,\ldots,p$ generated as i.i.d. standardized $T(8)$ random variables. In the simulation, we fix $n=1000$ and vary $p=50$ and $p=100$.  For each $m=1,\ldots, 100$, the coefficient matrices $A_m(i)=(1+m)^{-\beta}B_m(i)$, where  $\beta=1$, and $B_m(1)$ is an $R^{p\times p}$ block diagonal matrix. The $5\times 5$ diagonal blocks in $B_m(i)$ are fixed with i.i.d. $N(0,1)$ entries and all the other entries are $0$. 

We consider the number of abrupt changes is $\iota=2$ and $(nt^{(1)},nt^{(2)})=(300,650)$.  The matrix $A_0(i)$ is set to be a zero matrix for $i=1,2,\ldots,299$, while $A_0(i)=A_0(299)+\boldsymbol\alpha \boldsymbol \alpha^\top$, $i=300,301,\ldots,649$, and $A_0(i)=A_0(649)-\boldsymbol\alpha \boldsymbol \alpha^\top$, $i=650,651,\ldots,1000$, where the first $20$ entries in $\boldsymbol \alpha$ is taken to be a constant $\delta_0$ and the others are $0$.

We let the coefficient matrices $A_1(i)=\{a_{m,jk}(i)\}_{1\le j,k\le p}$ evolve at each time point such that two entries are soft-thresholded and another two elements increase. Specifically, at time $i$, we randomly select two elements from the support of $A_1(i)$, which are denoted as $\{a_{1,j_l^\star k_l^\star}(i)\}, l=1,2$ and that $a_{1,j^\star k^\star}(i)\ne 0$, and set them to $a^\star_{1,j_l^\star k_l^\star}(i)=\mbox{sign}(a_{1,j_l^\star k_l^\star}(i))(|a_{1,j_l^\star k_l^\star}(i)-0.05|)$. We also randomly select two elements from $A^\star_1(i)$ and increase their values by $0.03$.

Figure \ref{cov50} and Figure \ref{cov100} show the support of the true covariance matrices at $i=100,200,\ldots,900$.
\begin{figure}[htbp] 
   \centering
   \includegraphics[width=5in]{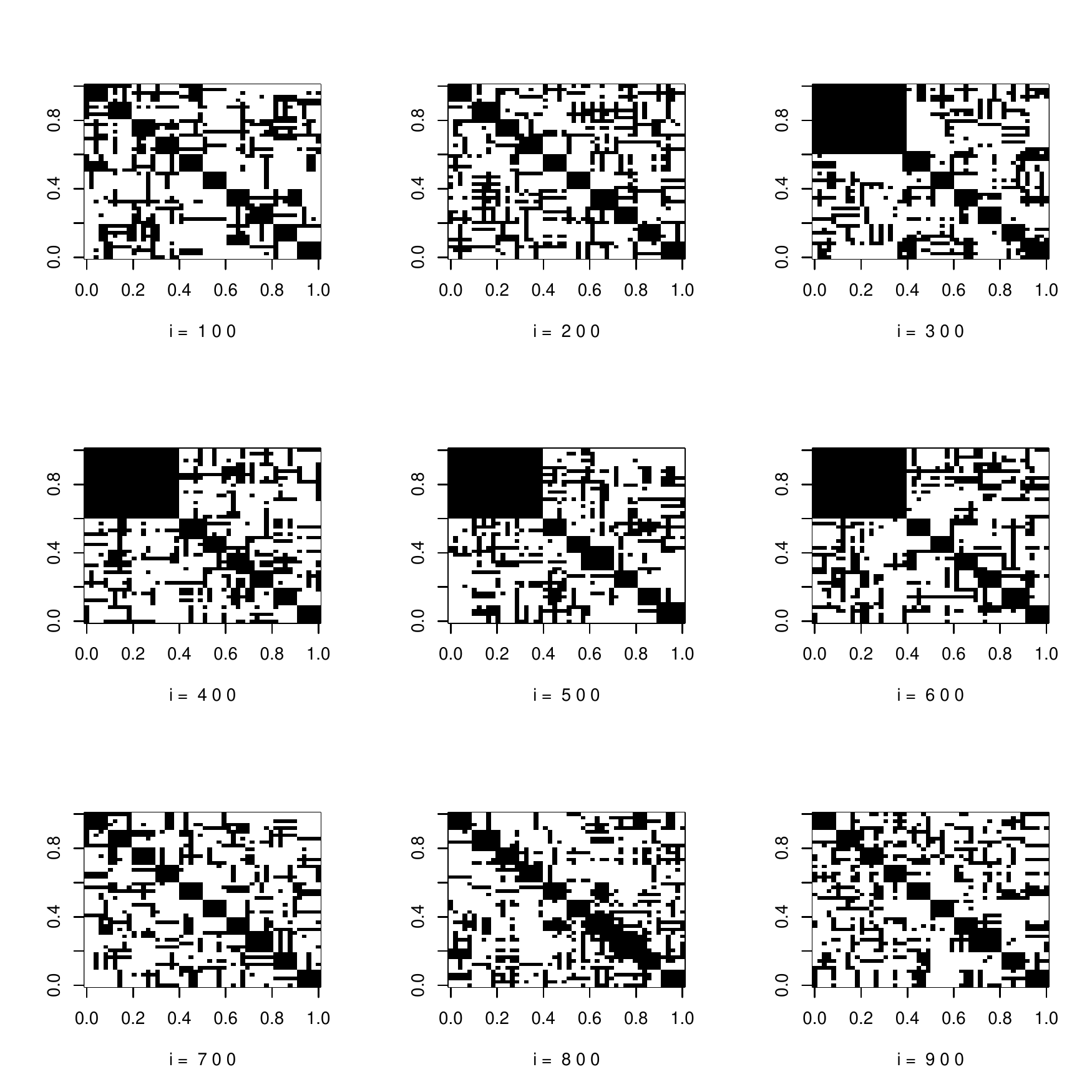} 
   \caption{Support of the true covariance matrices, $p=50$}
      \label{cov50}%  figure placement: here, top, bottom, or page

\end{figure}

\begin{figure}[htbp] %  figure placement: here, top, bottom, or page
   \centering
   \includegraphics[width=5in]{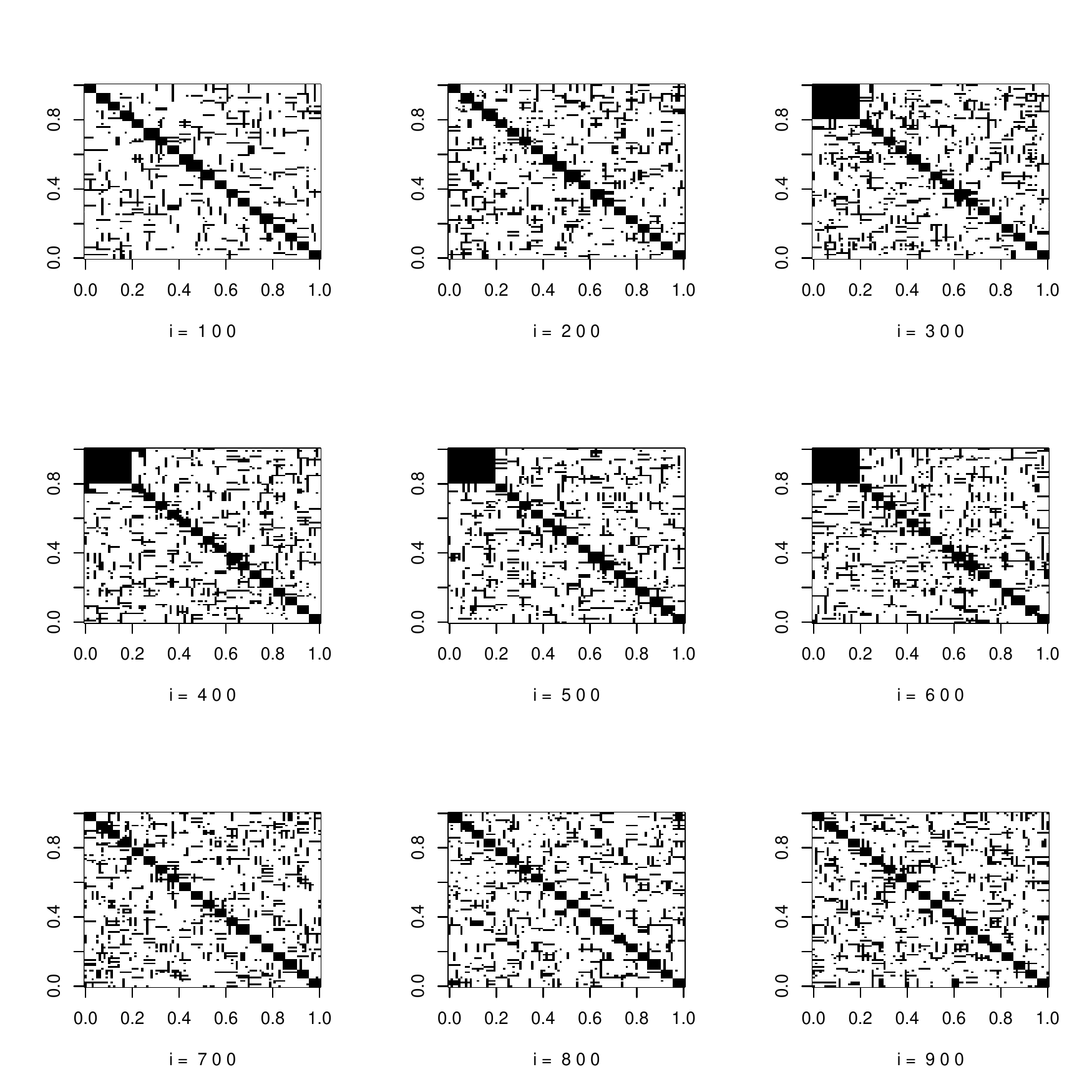} 
 \caption{Support of the true covariance matrices, $p=100$}
    \label{cov100}
\end{figure}

% The simulation is repeated for $500$ times.  

%The jump innovations $\bs\delta_i$ are $\bs 0$ for $0<i/n<t^{(1)}$ and i.i.d. $N(\bs 0,C^\star_l\times\Sigma_\delta)$ for $t^{(l)}\le i/n\le t^{(l+1)}, l=1,\ldots, \iota+1$, independent of $(\bs\epsilon_m)_{m\in \mathbb Z}$, with positive constants $C^\star_l$ that controls the jump magnitudes. We include various $C^\star_\delta$ from $0.5$ to $5$.

In detecting the change points, the cutoff value $\nu$ of detection is chosen as follows. After removing the neighborhood of detected change points, we obtain $\mathcal D_h^{(l)}$ by ordering $\mathcal D_h^{(l)},\ldots \mathcal D_h^{(\mathfrak l)}$, where $\mathfrak l$ is obtained from (\ref{eqn:early_stop}) with $\nu=0$. For $l=1,2,\ldots,\mathfrak l-1$, compute
\[
\mathcal R_h^{(l)}={\mathcal D_h^{(l)}\over \mathcal D_h^{(l+1)}}.
\]
We let $\hat\iota=\arg\max_{0\le l\le \mathfrak l-1 }\mathcal R_h^{(l)}$ and set $\nu=\mathcal D_h^{(\hat\iota)}$.

We report the number of estimated jumps and the average absolute estimation error, where the average absolute estimation error is the mean of the distance between the estimated change points and the true change points. As is shown from Table \ref{table:distance} and Table \ref{table:number}, there is an apparent improvement in the estimation accuracy as the jump magnitude increases and dimension decreases. The detection is relatively robust to the choice of bandwidth.

We evaluate the support recovery performance of the time-varying CLIME at the lattice $100,200,\ldots,900$ with $\lambda=0.02,0.06,0.1$. We take the uniform kernel function and the bandwidth is fixed as $0.2$. At each time point $t_0$, two quantities are computed: sensitivity and specificity, which are defined as:
\begin{align*}
\mbox{sensitivity}&=\frac{\sum_{1\le j,k\le p}\mathbb{I}\{\hat g_{jk}(t_0;u)\ne 0, g_{jk}(t_0;u)\ne 0\}}{\sum_{1\le j,k\le p}\mathbb{I}\{g_{jk}(t_0;u)\ne 0\}},\\
\mbox{specificity}&=\frac{\sum_{1\le j,k\le p}\mathbb{I}\{\hat g_{jk}(t_0;u)= 0, g_{jk}(t_0;u)= 0\}}{\sum_{1\le j,k\le p}\mathbb{I}\{ g_{jk}(t_0;u)= 0\}}.
\end{align*}
We plot the Receiver Operating Characteristic (ROC) curve, that is, sensitivity against 1-specificity. From Figure \ref{fig:roc_p50} and Figure \ref{fig:roc_p100} we observe that, due to a screening step, the support recovery is robust to the choice of $\lambda$, except at the change points, where a non-negligible estimation error of the covariance matrix is induced and the overall estimation is less accurate. As the effective dimension of the network remains the same at $p=50$ and $p=100$ by the construction of the coefficient matrix $A_m(i)$, there is no significant difference in the ROC curves at different dimensions.

\begin{table}[htp]
\caption{Average distance.}
\begin{center}
\begin{tabular}{cccccccc}
\hline\hline
%&&\multicolumn{6}{c}{Bandwidth}\\
%\hline
&bandwidth &0.14&0.16&0.18&0.2&0.22&0.24\\
\hline
\multirow{3}{*}{$p=50$}
&$\delta_0=1$&23.4&21.0&17.47&16.6&14.7&16.5\\
&$\delta_0=2$&7.4&6.9&8.3&8.1&7.2&6.3\\
\hline
\multirow{3}{*}{$p=100$}
&$\delta_0=1$&37.2&30.1&26.4&25.5&21.2&21.3\\
&$\delta_0=2$&7.8&8.2&9.9&6.9&8.9&7.6\\
\hline\hline
\end{tabular}
\end{center}
\label{table:distance}
\end{table}%

\begin{table}[htp]
\caption{Number of estimated change points.}
\begin{center}
\begin{tabular}{cccccccc}
\hline\hline
%&\multicolumn{5}{c}{Bandwidth}\\
%\hline
&Bandwidth &0.14&0.16&0.18&0.2&0.22&0.24\\
\hline
\multirow{3}{*}{$p=50$}
&$\delta_0=1$&2.38&2.16&1.99&2.00&2.00&2.00\\
&$\delta_0=2$&2.46&2.31&2.00&2.00&2.00&2.00\\
\hline
\multirow{3}{*}{$p=100$}
&$\delta_0=1$&2.25&2.09&1.99&1.99&2.00&2.00\\
&$\delta_0=2$&2.38&2.19&2.00&2.00&2.00&2.00\\
\hline\hline
\end{tabular}
\end{center}
\label{table:number}
\end{table}%

\begin{figure}[htbp] %  figure placement: here, top, bottom, or page
   \centering
   \includegraphics[width=5in]{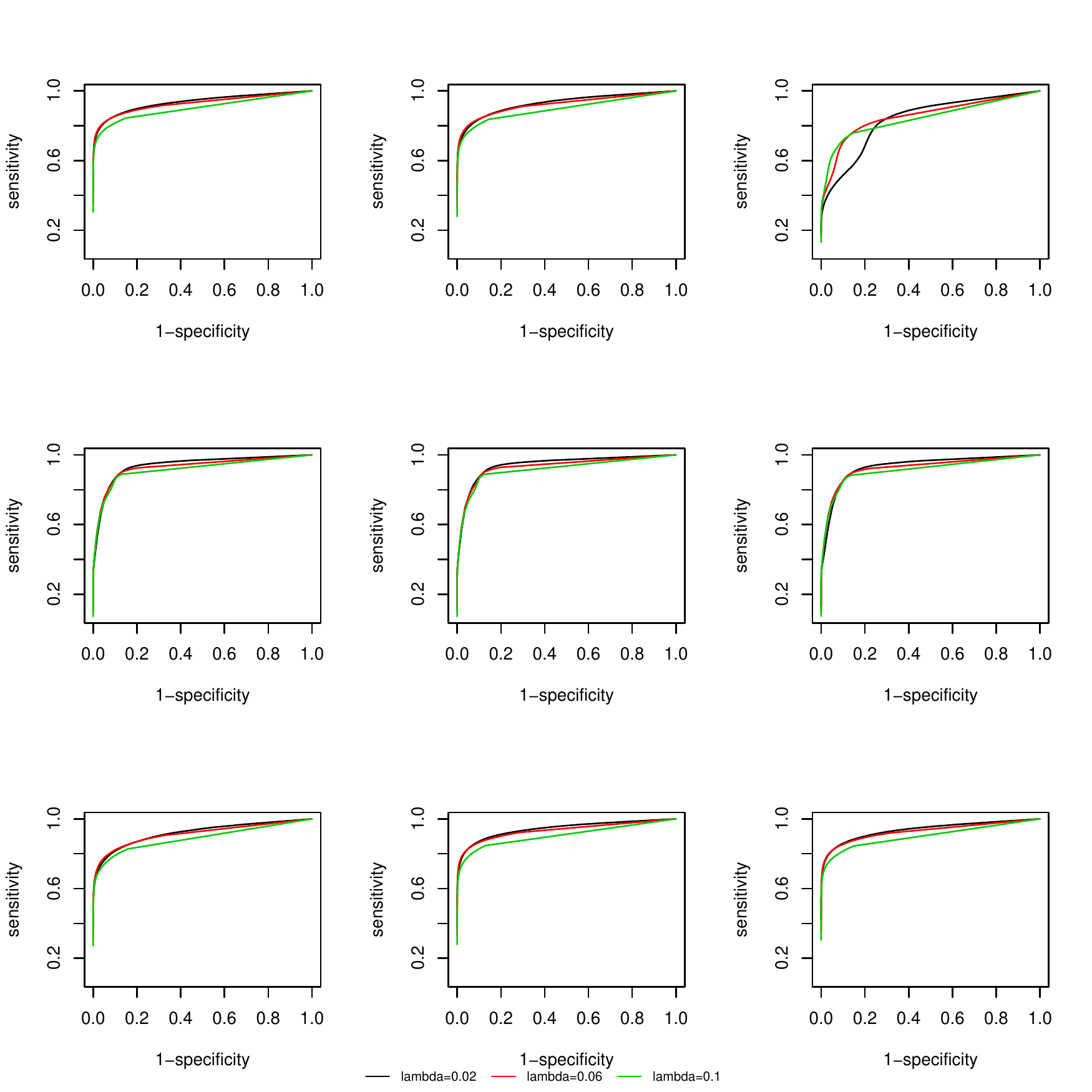} 
   \caption{ROC curve of the time-varying CLIME, $p=50$}
   \label{fig:roc_p50}
\end{figure}

\begin{figure}[htbp] %  figure placement: here, top, bottom, or page
   \centering
   \includegraphics[width=5in]{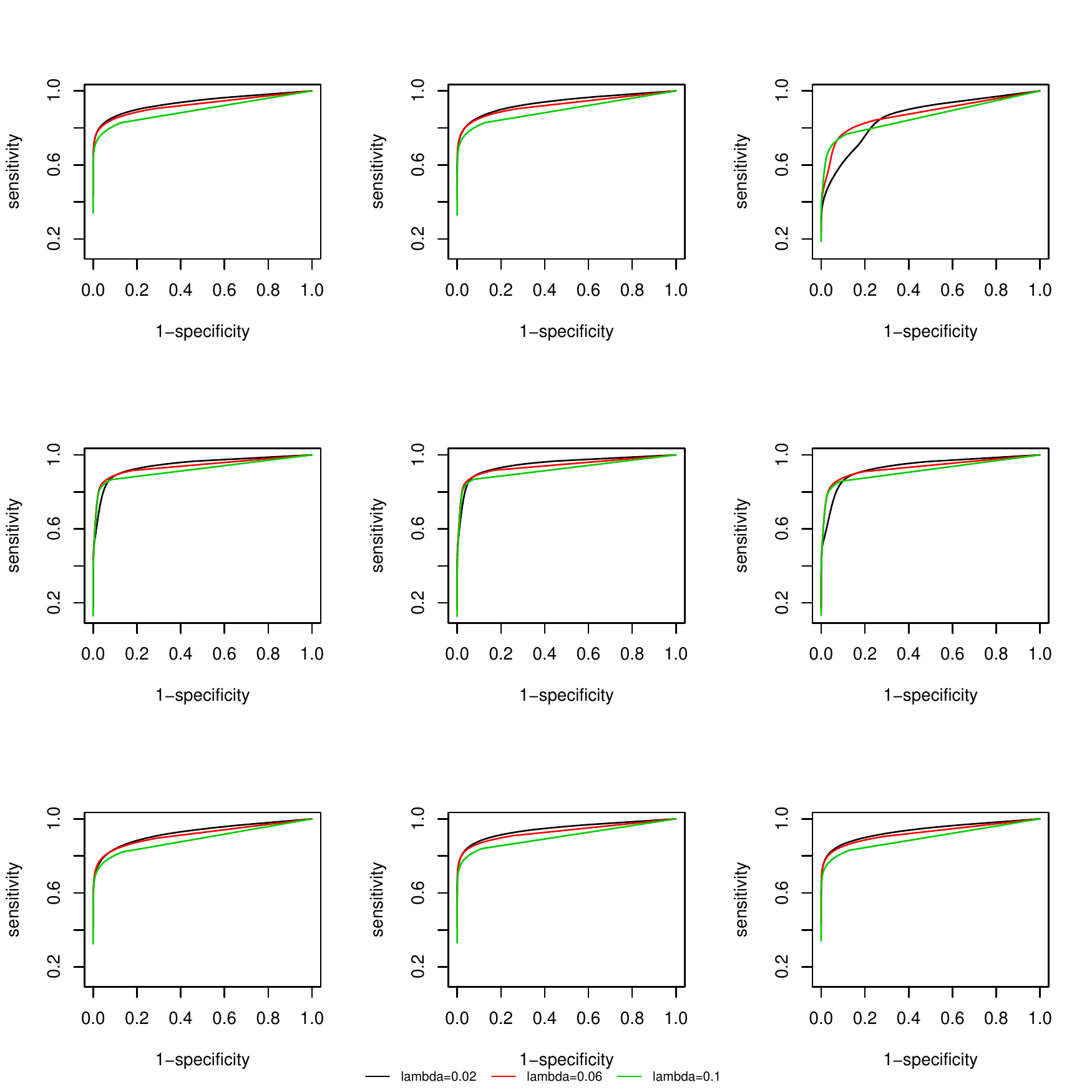} 
   \caption{ROC curve of the time-varying CLIME, $p=100$}
      \label{fig:roc_p100}
\end{figure}

\section{A real data application}
\label{sec:real-data-analysis_network}

Understanding the interconnection among financial entities and how they vary over time provides investors and policy makers with insights in risk control and decision making. \cite{allen2009networks} presents a comprehensive study of the applications of network theory in financial systems. In this section, we apply our method to a real financial dataset from Yahoo! Finance (\texttt{finance.yahoo.com}). The data matrix contains daily closing prices of 420 stocks that are always in the S\&P 500 index between January 2, 2002 through December 30, 2011. In total, there are $n=2519$ time points. We select 100 stocks with the largest volatility and consider their log-returns; that is, for $j=1,\dots,100$,
\begin{equation*}
X_{ij} = \log\rbr{p_{i+1,j} /  p_{ij}},
\end{equation*}
where $p_{ij}$ is the daily closing price of the stock $j$ at time point $i$.  %For each stock, the log-returns $X_{ij}$ are then standardized to zero mean and unit variance. 
We first compute the statistic (\ref{eqn:D_nonstationary}) and (\ref{eqn:change-point-estimator-infty-nonstationary1}) for the change point detection. We look at the top three statistics for different bandwidths. For bandwidth $k=n^{-1/5}=0.21$, we rank the test statistic and find that the location for the top change point is: February 07, 2008 ($n_{\hat{s}_1} = 1536$), which is shown in Figure~\ref{fig:D_max}. The detected change point is quite robust to a variety of choices of bandwidth. Our result is partially consistent with the change point detection method in \cite{MR2572452}. In particular, the two breaks in 2006 and 2007 were also found in \cite{MR2572452} and it is conjectured that the 2007 break may be associated to the U.S. house market collapse. Meanwhile, it is interesting to observe the increased volatility before the 2008 financial crisis.

\begin{figure}[h!]
    \centering
    \includegraphics[width=4in]{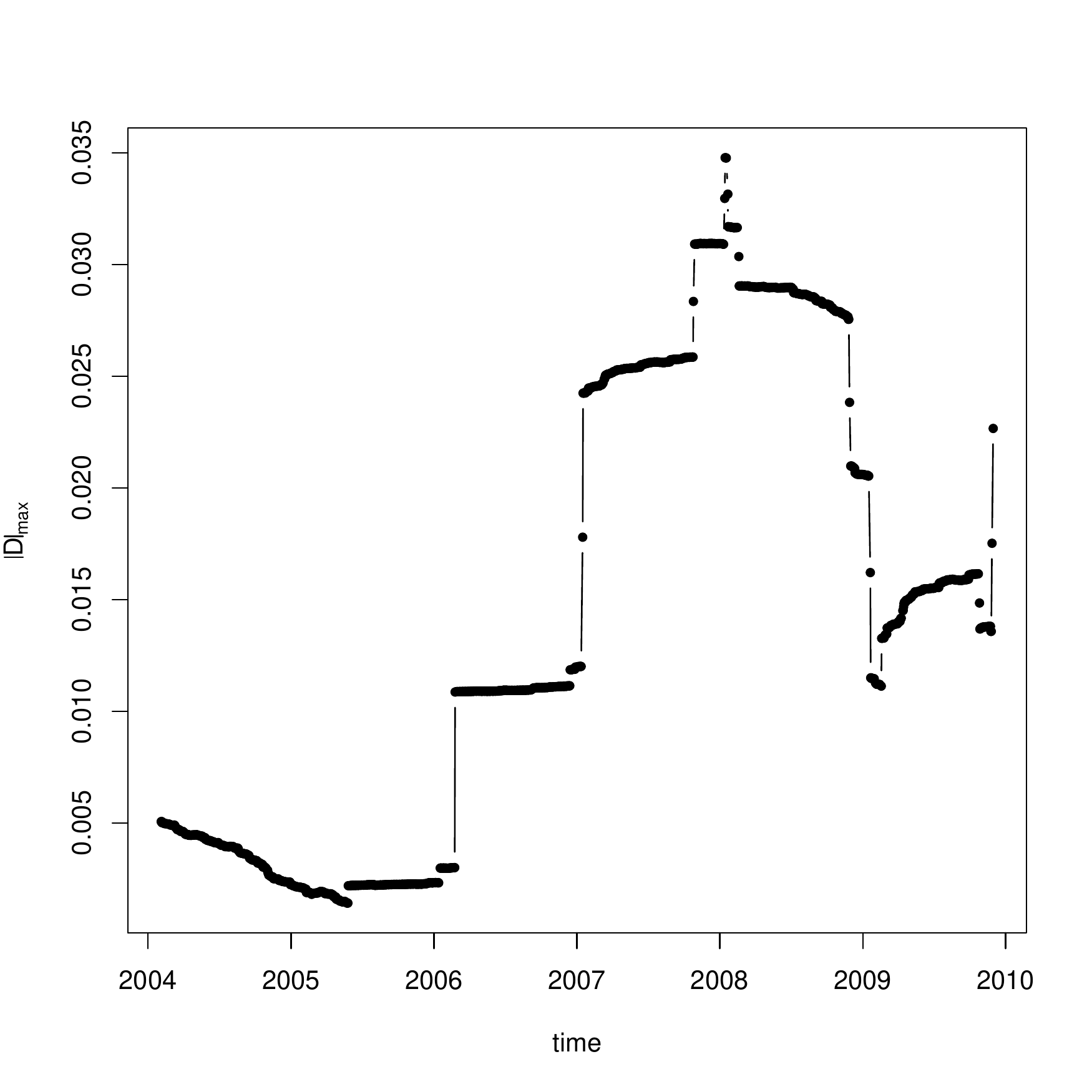}
    \caption{Break size $|D_s|_\infty$. From February 4, 2004, to November 30, 2009. %\textcolor{red}{Not sure if we want to show this plot...}
    }
    \label{fig:D_max}
\end{figure}

Next, we estimate the time-varying networks before and after the change point at May 26, 2006 with the largest jump size. Specifically, we look at four time points at: 813, 828, 888, and 903, corresponding to March 23, 2006 April 13, 2006, July 11, 2006, and August 1, 2006. We use tv-CLIME (\ref{eqn:CLIME}) with the Epanechnikov kernel with the same bandwidth as in the change point detection to estimate the networks at the four points. Optimal tuning parameter $\lambda$ is automatically selected according to the stability approach \cite{liuroederwasserman2010}. The following matrix shows the number of different edges at those four time points. It is observed that time the first two time points (813 and 828) and the last two (888 and 903) have higher similarity than across the change point at time 858. The estimated networks are shown in Figure \ref{fig:time-varying-networks}. Networks in the first and second row are estimated before and after the estimated change point at time 858, respectively. It is observed that at each time point the companies in the same section tend to be clustered together such as companies in the \texttt{Energy} section: OXY, NOV, TSO, MRO and DO (highlighted in cyan). In addition, the distance matrix of estimated networks is estimated as

\begin{equation*}
\left(
\begin{array}{cccc}
0 & 332 & 350 & 396 \\
332 &   0 & 394 & 428 \\
350 & 394  &  0 & 234 \\
396 & 428 & 234 & 0 \\
\end{array}
\right).
\end{equation*}

\begin{figure}[htp] %  figure placement: here, top, bottom, or page
   \centering
  \subfigure[Time 813.] {\label{subfig:time_813}\includegraphics[scale=0.3]{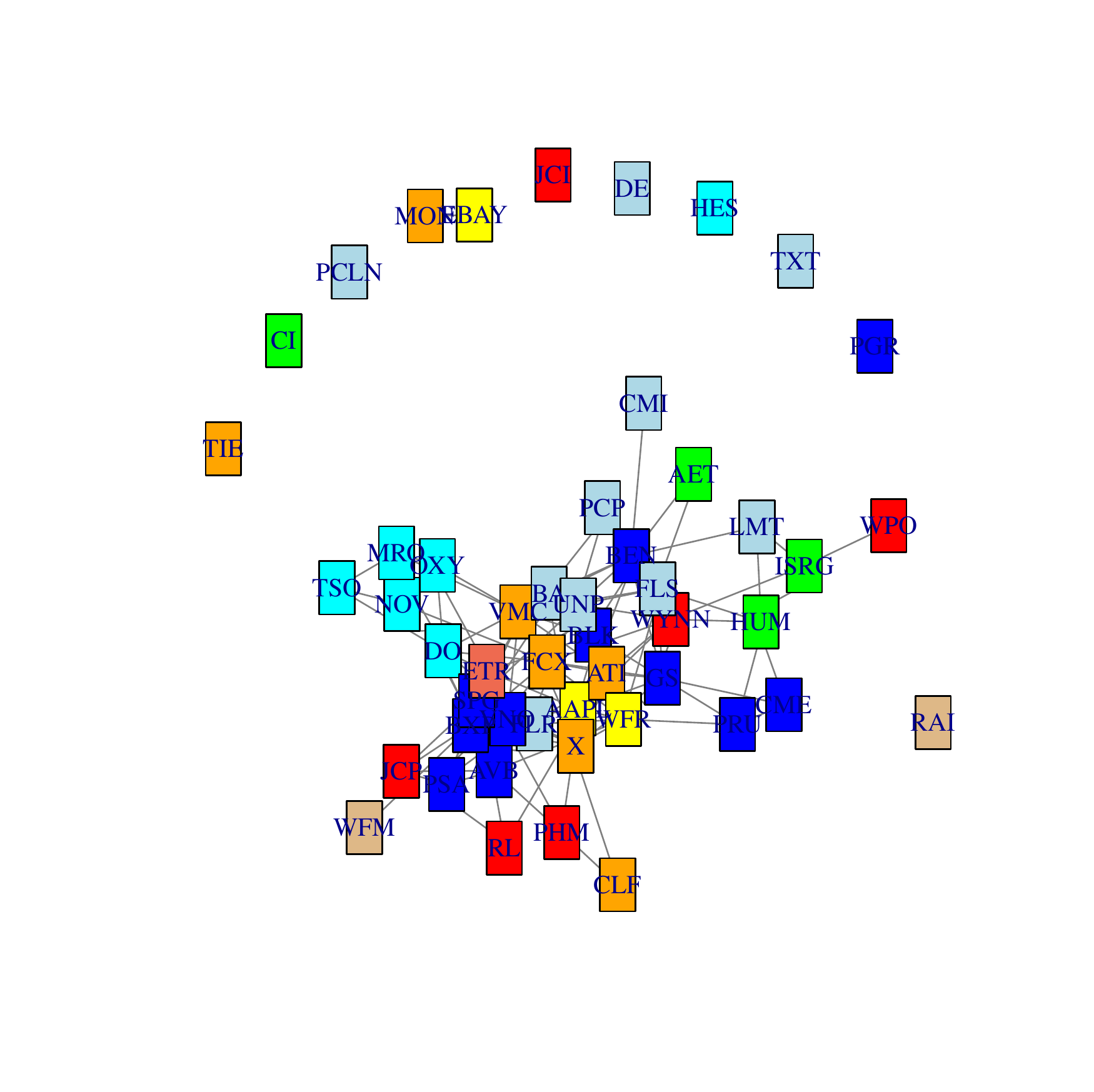}}
	\subfigure[Time 828.] {\label{subfig:time_828}\includegraphics[scale=0.3]{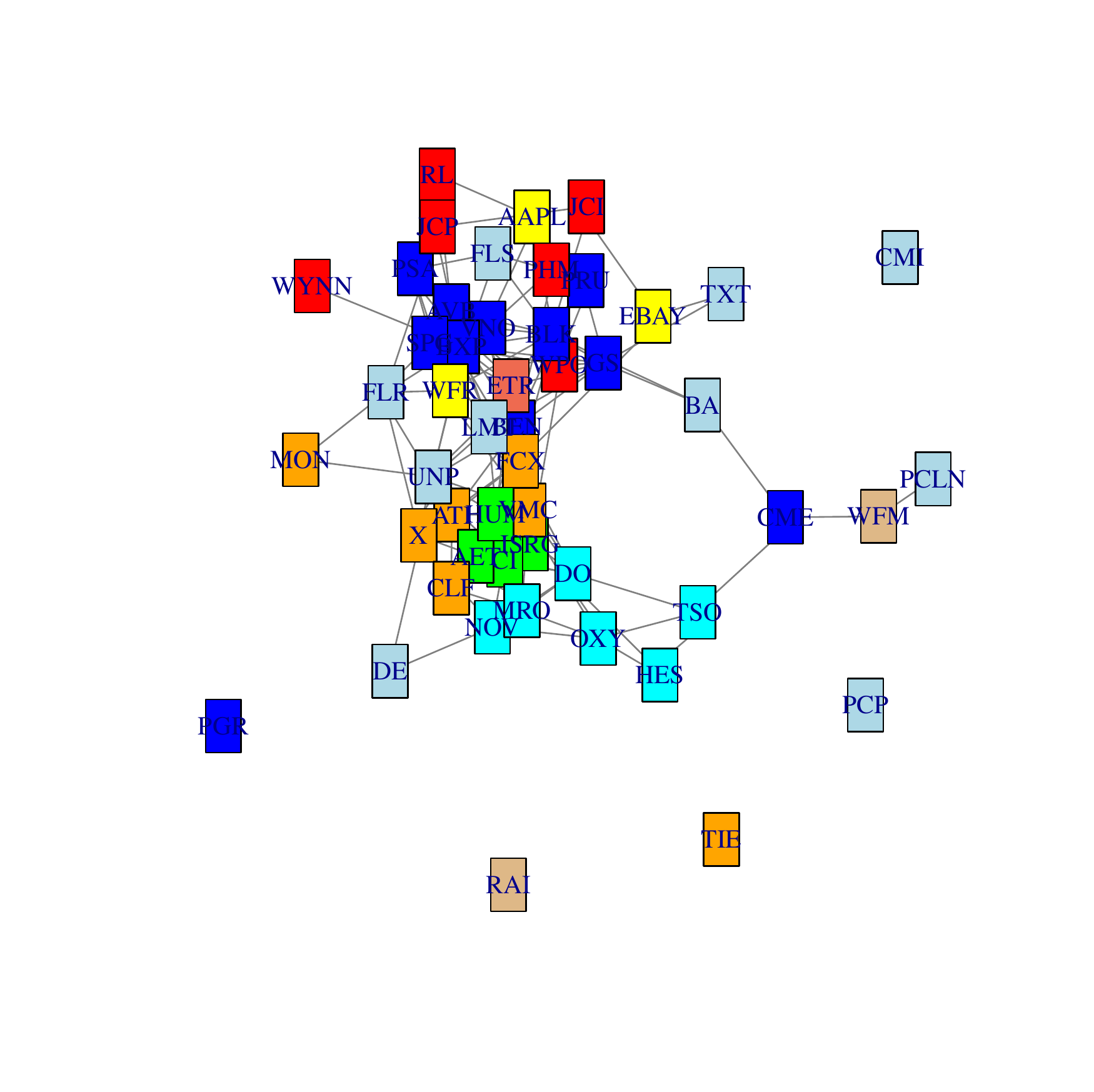}} \\
	\subfigure[Time 888.] {\label{subfig:time_888}\includegraphics[scale=0.3]{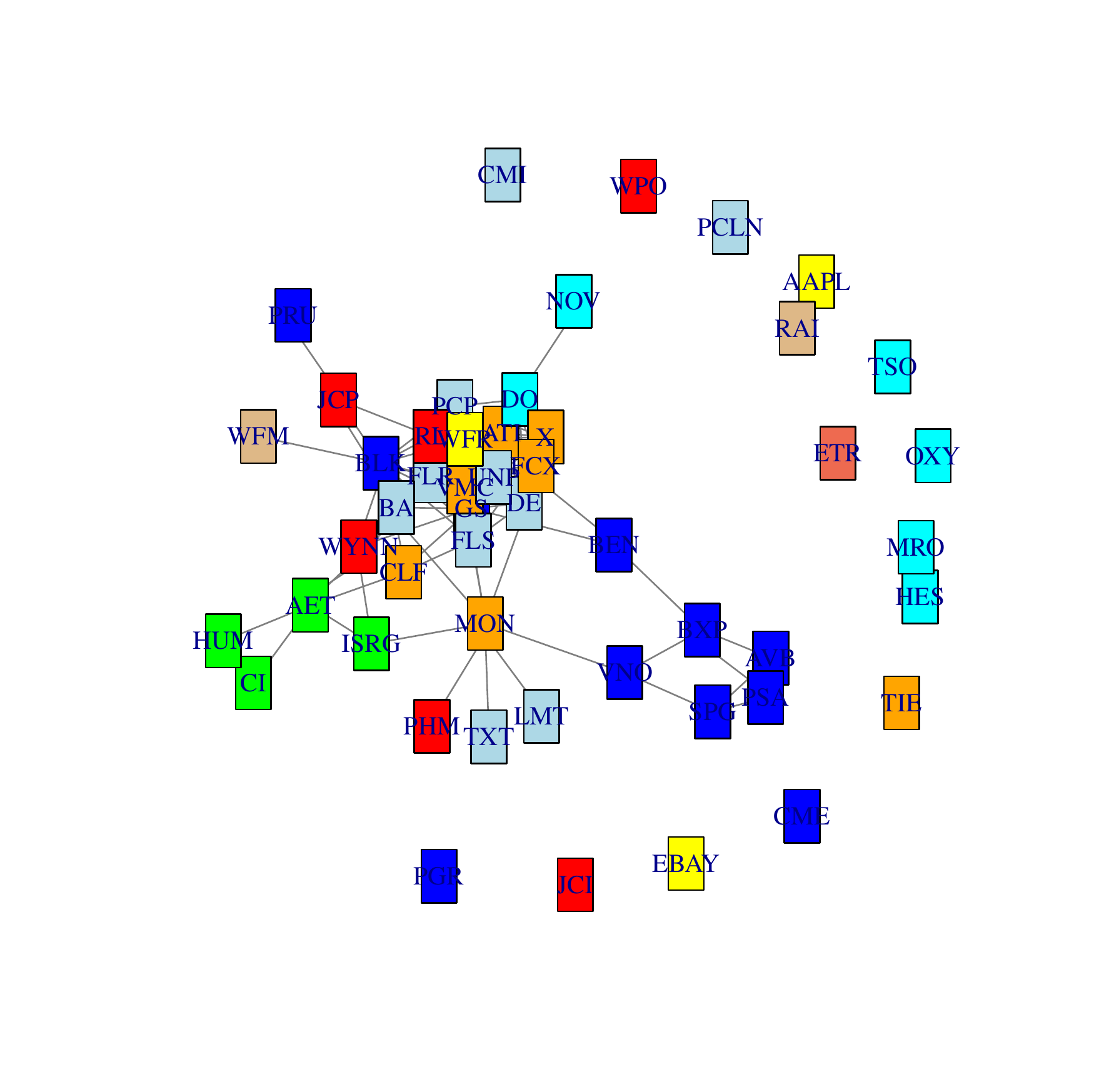}}
	\subfigure[Time 903.] {\label{subfig:time_903}\includegraphics[scale=0.3]{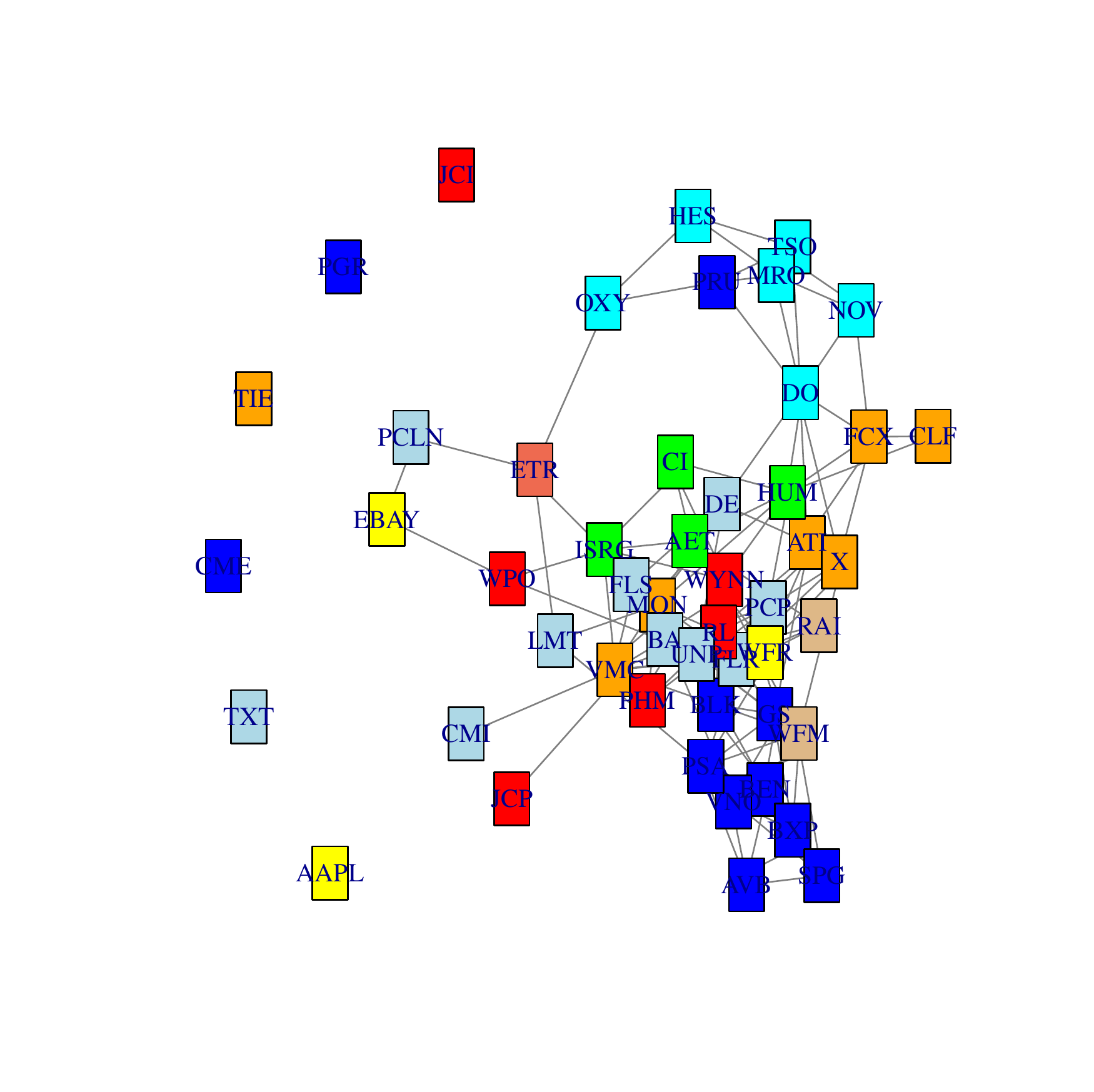}}\\
   \caption{Estimated networks at time points 813, 828, 888 and 903, corresponding to March 23, 2006, April 13, 2006, July 11, 2006, and August 1, 2006. Colors correspond to the nine sections in the S\&P dataset.}
   \label{fig:time-varying-networks}
\end{figure}

\section{Proof of main results}
\label{sec:proof}

\subsection{Preliminary lemmas}
\begin{lem}\label{lem:nagaev}
Let $(Y_i)_{i\in Z}$ be a sequence that admits (\ref{eqn:data_local}). Assume $Y_i\in \mathcal L^q$ for $i=1,2,\dots$, and the dependence adjusted norm (DAN) of the corresponding underlying array $(Y_i^\circ(t))$ satisfies $\|Y_\cdot\|_{q,A}<\infty$ for $q>2$ and $A>0$. Let $(\omega(t,t_i))_{i=1}^n$ be defined in (\ref{eqn:kernel weight}) and suppose that the kernel function $K(\cdot)$ satisfies Assumption \ref{assumption:kernel}. Denote $\varpi_{q,A}(n) = n, n(\log n)^{1+2q}, n^{q/2-Aq}$ if $A > 1/2-1/q$, $A = 1/2-1/q$, and $0 < A < 1/2-1/q$, respectively. Then there exist constants $C_1,C_2$ and $C_3$ independent of $n$, such that for all $x>0$,
\begin{align}
\label{eqn:dev-Y}
\sup_{t\in (0,1)}\P\rbr{\rBR{\sum_{i=1}^nw(t,t_i)\big(Y_i-\E(Y_i)\big)}>x}\le C_1{\varpi_{q,A}(B_n)\RBR{Y_\cdot}_{q,A}^q\over B_n^{q} x^q}+C_2\exp\rbr{-C_3 B_n x^2\over  \RBR{Y_\cdot}_{2,A}^2}.
\end{align}
\begin{align}
\label{eqn:max-dev-Y}
\P\rbr{\sup_{t\in (0,1)}\rBR{\sum_{i=1}^nw(t,t_i)\big(Y_i-\E(Y_i)\big)}>x}\le C_1{\varpi_{q,A}(n)\RBR{Y_\cdot}_{q,A}^q\over B_n^{q} x^q}+C_2\exp\rbr{-C_3 B_n^2 x^2\over  n\RBR{Y_\cdot}_{2,A}^2}.
\end{align}
\end{lem}
\begin{proof}
Let $S_i=\sum_{j=1}^i \big(Y_i-\E (Y_i)\big)$. Note that 
\begin{align}\nonumber
\sup_{t\in (0,1)}\rBR{\sum_{i=1}^nw(t,t_i)Y_i}&=\sup_{t\in (0,1)}\rBR{\sum_{i=1}^n w(t,t_i)(S_{i}-S_{i-1})}\\\nonumber
&\le \sup_t\rBR{\sum_{i=1}^{n-1}\rbR{\big(w(t,t_i)-w(t,t_{i+1})\big)S_i}}+\sup_t\rBR{w(t,1)S_n}\\\nonumber
&\lesssim B_n^{-1}\max_{1\le i\le n}|S_i|,
\end{align}
where the last inequality follows from the fact that $\sup_t\sum_{i=1}^n|w(t,t_i)-w(t-t_{i+1})|\asymp B_n^{-1}$ due to Assumption \ref{assumption:kernel}. 

To see (\ref{eqn:max-dev-Y}), it suffices to show
\begin{align}\label{eqn:SM121944}
\P\rbr{\max_{1\le i\le n}|S_i|>x}\le C_1{\varpi_{q,A}(n)\RBR{Y_\cdot}_{q,A}^q\over  x^q}+C_2\exp\rbr{-C_3 x^2\over  n\RBR{Y_\cdot}_{2,A}^2}.
\end{align}

Now we develop a probability deviation inequality for $\max_{1\le i\le n}|\sum_{j=1}^i\alpha_jY_j|$, where  $\alpha_j\ge 0$, $1\le j\le n$ are constants such that $\sum_{1\le j\le n}\alpha_j=1$. 
% For (\ref{eqn:dev-Y}), we take $\alpha_j=w(t,t_j)$ and for (\ref{eqn:max-dev-Y}) we take $\alpha_j=B_n^{-1}$. 
Denote $\mathcal P_0(Y_i)=\E(Y_i|\varepsilon_i)-\E(Y_i)$ and
\[
\mathcal P_k(Y_i)=\E(Y_i|\varepsilon_{i-k},\ldots,\varepsilon_i)-\E(Y_i|\varepsilon_{i-k+1},\ldots,\varepsilon_i). 
\]
Then we can write 
\begin{align}\label{eqn:SM121947}
\max_{1\le i\le n}|\sum_{j=1}^i\alpha_jY_j|
&\le \max_{1\le i\le n}|\sum_{j=1}^i\alpha_j \Proj_0(Y_j)|
+\max_{1\le i\le n}|\sum_{k=1}^n\sum_{j=1}^i\alpha_j \Proj_k(Y_j)|\\\nonumber
&\qquad+\max_{1\le i\le n}|\sum_{k=n+1}^\infty \sum_{j=1}^i\alpha_j \Proj_k(Y_j)|.
\end{align}
Note that $(\Proj_0(Y_j))_{j\in\mathbb Z}$ is an independent sequence. By Nagaev's inequality and Ottaviani's inequality, we have that 
\begin{align}\label{eqn:SM121948}
\Prob(\max_{1\le i\le n}|\sum_{j=1}^i\alpha_j \Proj_0(Y_j)|\ge x)
&\lesssim {\sum_{j=1}^n\alpha_j^q \RBR{\Proj_0(Y_j)}_q^q\over x^q}+\exp\big(-{C_3 x^2\over \sum_{j=1}^n\alpha_j^2 \|\Proj_0(Y_j)\|_2^2}\big)\\\nonumber
&\lesssim\frac{\sum_{j=1}^n\alpha_j^q}{x^q\|Y_j\|_q}+\exp\big(-C_3 {x^2\over \sum_{j=1}^n\alpha_j^2}\big),
\end{align}
where the last inequality holds because $\|\Proj_0(Y_j)\|_q\le 2\|Y_j\|_q$  by Jensen's inequality. 
Since $\sum_{j=i+1}^\infty\alpha_j \Proj_k(Y_j)$ is a martingale difference sequence with respect to $\sigma(\varepsilon_{i+1-k},\varepsilon_{i+2-k},\ldots)$, we have that $|\sum_{k=1+n}^\infty\sum_{j=i+1}^n\alpha_j\Proj_k(Y_j)|$ is a non-negative sub-martingale. Then by Doob's inequality and Burkholder's inequality, we have
\begin{align}\nonumber
&\Prob\big({\max_{1\le i\le n}|\sum_{k=n+1}^\infty\sum_{j=1}^i\alpha_j\Proj_k(Y_j)|\ge x}\big)\\\nonumber
&\le \Prob\big({|\sum_{k=n+1}^\infty\sum_{j=1}^n\alpha_j\Proj_k(Y_j)|\ge {x\over2}}\big)+\Prob\big({\max_{1\le i\le n}|\sum_{k=n+1}^\infty\sum_{j=1+i}^{n}\alpha_j\Proj_k(Y_j)|\ge {x\over2}}\big)\\\nonumber
&\lesssim\frac{\RBR{\sum_{k=1+n}^\infty\sum_{j=1}^n\alpha_j\Proj_k(Y_j)}_q^q}{x^q}\\\label{eqn:SM121951}
&\lesssim\frac{(\sum_{j=1}^n \alpha_j^2)^{q/2}\Theta_{n,q}^q}{x^q}
\le \frac{\Theta_{n,q}^qn^{q/2-1}\sum_{j=1}^n \alpha_j^q}{x^q}.
\end{align}
Now we deal with the term $\max_{1\le i\le n}|\sum_{k=1}^n\sum_{j=1}^i \alpha_j\Proj_k(Y_j)|$. Define $a_m=\min(2^m,n)$ and $M_n=\lceil\log n/\log 2\rceil$. Then
\begin{align}\label{eqn:TM101953}
\max_{1\le i\le n}\big|\sum_{k=1}^n\sum_{j=1}^i \alpha_j\Proj_k(Y_j)\big|
\le \sum_{m=1}^{M_n}\max_{1\le i\le n}\big|\sum_{l=1}^{\lceil i/a_m\rceil}\sum_{j=1+(l-1)a_m}^{\min(la_m,i)}\sum_{k=1+a_{m-1}}^{a_m}\alpha_j\Proj_k(Y_j)\big|.
\end{align}
Let $\mathcal A_{odd}=\{1\le l\le \lceil i/a_m\rceil,l \mbox{ is odd}\}$ and $\mathcal A_{even}=\{1\le l\le \lceil i/a_m\rceil,l \mbox{ is even}\}$. We have
\begin{align*}
\Prob\big(\max_{1\le i\le n}\big|\sum_{l=1}^{\lceil i/a_m\rceil}Z_{l,m,i}\big|\ge x\big)
\le \Prob\big(\max_{1\le i\le n}\big|\sum_{\mathcal A_{odd}}Z_{l,m,i}\big|\ge x/2\big)
+\Prob\big(\max_{1\le i\le n}\big|\sum_{\mathcal A_{even}}Z_{l,m,i}\big|\ge x/2\big),
\end{align*}
where we have that $Z_{l,m,i}:=\sum_{j=1+(l-1)a_m}^{\min(la_m,i)}\alpha_j\Proj_{a_{m-1}}^{a_m}(Y_j)$ is independent of $Z_{l+2,m,i}$ for $1\le l\le \lceil i/a_m\rceil, 1\le m\le M_n, 1\le i\le n$, as $\Proj_{a_{m-1}}^{a_m}(Y_j):=\sum_{k=1+a_{m-1}}^{a_m}\Proj_k(Y_j)$ is $a_m$-dependent. Therefore, we can apply   Ottaviani's inequality and Nagaev's inequality for independent variables. As a consequence, 
\[
\Prob\big(\max_{1\le i\le n}\big|\sum_{l=1}^{\lceil i/a_m\rceil}Z_{l,m,i}\big|\ge x\big)
\lesssim \frac{\sum_{1\le l\le \lceil n/a_m\rceil}\|Z_{l,m,n}\|_q^q}{x^q}
+\exp\big(-{C_3 x^2\over \sum_{1\le l\le \lceil n/a_m\rceil}\|Z_{l,m,n}\|_2^2}\big).
\]
Again, by Burkholder's inequality, we have that for $q\ge 2$,
\begin{align*}
\|Z_{l,m,n}\|_q
&\le \sum_{k=1+a_{m-1}}^{a_m}\|\sum_{j=1+(l-1)a_m}^{\min(la_m,n)}\alpha_j\Proj_k(Y_j)\|_q\\
&\lesssim (\sum_{j=1+(l-1)a_m}^{\min(la_m,n)}\alpha_j^2)^{1/2}(\Theta_{a_{m-1}}-\Theta_{a_m}).
\end{align*}
Note $\sum_{j=1+(l-1)a_m}^{\min(la_m,n)}\alpha_j^2\le a_m^{(q-2)/q}(\sum_{j=1+(l-1)a_m}^{\min(la_m,n)}\alpha_j^q)^{2/q}$.  Let $\tau_m=m^{-2} /\sum_{m=1}^{M_n} m^{-2}$, and we have $\tau_m\asymp m^{-2}$ as $1\le \sum_{m=1}^{M_n} m^{-2}\le \pi^2/6$. In respect to (\ref{eqn:TM101953}), we have that
\begin{align}\label{eqn:SM122004}
\Prob\big(\max_{1\le i\le n}\big|\sum_{k=1}^n\sum_{j=1}^i\Proj_k(Y_j)\big|\ge x\big)
&\le \sum_{m=1}^{M_n} \Prob\big(\max_{1\le i\le n}\big|\sum_{l=1}^{\lceil i/a_m\rceil}Z_{l,m,i}\big|\ge \tau_m x\big)\\\nonumber
&\hskip-2cm
\lesssim \frac{\sum_{i=1}^n \alpha_j^q}{x^q}\|Y_\cdot\|_{q,A}^q\sum_{m=1}^{M_n} \tau_m^{-q}a_m^{(1/2-A)q-1}
+ \sum_{m=1}^{M_n}\exp\big(-\frac{C_3x^2\tau_m^2 a_m^{2A}}{\sum_{j=1}^n\alpha_j^2\|Y_\cdot\|_{2,A}^2}\big).
\end{align}
Note $\sum_{m=1}^{M_n} \tau_m^{-q}a_m^{(1/2-A)q-1}\asymp n^{-1}\varpi_{q,A}(n)$, and
\[
\sum_{m=1}^{M_n}\exp\big(-\frac{C_3x^2\tau_m^2 a_m^{2A}}{\sum_{j=1}^n\alpha_j^2\|Y_\cdot\|_{2,A}^2}\big) \lesssim \exp\big(-\frac{C_3x^2}{\sum_{j=1}^n\alpha_j^2 \|Y_\cdot\|_{2,A}^2}\big).
\]
Combining (\ref{eqn:SM121947}), (\ref{eqn:SM121948}), (\ref{eqn:SM121951}) and (\ref{eqn:SM122004}), we obtain 
\begin{align}\nonumber
&\P\big(\max_{1\le i\le n}\big|{\sum_{j=1}^i\alpha_j\big(Y_j-\E (Y_j)\big)}\big|>x\big)\\
\label{eqn:WJ202046}
&\qquad\le C_1{\varpi_{q,A}(n)\sum_{j=1}^n\alpha_j^q\RBR{Y_\cdot}_{q,A}^q\over nx^q}+C_2\exp\big({-C_3 x^2\over \sum_{j=1}^n\alpha_j^2 \RBR{Y}_{2,A}^2}\big). 
\end{align}
Now we have (\ref{eqn:SM121944}) by taking $\alpha_j=n^{-1}$ for $j=1,\ldots, n$. 
Note that since $K(\cdot)$ has bounded support, for any given $t\in [b,1-b]$, we have
\begin{align*}
\P\big(\big|\sum_{i=1}^n w(t,t_i)(Y_i-\E Y_i)\big|>x\big)\le \P\big(\big|\sum_{i=-B_n}^{B_n}w(t,t_{tn+i})(Y_{tn+i}-\E Y_{tn+i})\big|>x\big)\\
\le C_1{\varpi_{q,A}(B_n)\sum_{i=-B_n}^{B_n}w(t,t_{tn+i})^q\RBR{Y_\cdot}_{q,A}^q\over B_nx^q}+C_2\exp\big({-C_3 x^2\over\sum_{i=-B_n}^{B_n}w(t,t_{tn+i})^2 \RBR{Y_\cdot}_{2,A}^2}\big). 
\end{align*}
 Therefore (\ref{eqn:dev-Y}) follows from (\ref{eqn:WJ202046}) by taking $\alpha_j=w(t,tn+j)$ and note that 
 Note that for any $t\in [b,1-b]$, $\sum_{i=-B_n}^{B_n}w(t,t_{tn+i})^\beta\asymp B_n^{1-\beta}$ for a constant $\beta\ge 2$.
\end{proof}

\begin{lem}
\label{lemma:Y_i,jk_large-dev}
Suppose $(X_{ij})_{i\in \mathbb Z, 1\le j\le p}$ satisfy Assumption \ref{assumption:dep}. Also let Assumption \ref{assumption:kernel} hold. Let $\varpi_{q,A}(n)$ be defined as in Lemma \ref{lem:nagaev}. %Recall from (\ref{eqn:moment_dan_vecproc}) that $\nu_{2q} = \max_{1\le i\le n}\max_{1\le j\le n} \| X_{ij}\|_{2q}$, $M_{X,q}=(\sum_{j=1}^p \|X_{\cdot,j}\|_{2q,A}^q)^{1/q}$ and $N_X =   \max_{1\le j \le p} \|X_{\cdot,j}\|_{4,A}$. 
Then there exist constants $C_1,C_2$ and $C_3$ independent of $n$ and $p$, such that for all $x > 0$, we have
\begin{eqnarray}
\nonumber
&&\sup_{t\in (0,1)}\Prob\big( \big|\sum_{i=1}^n \omega(t,t_i) \big(\vX_i\vX_i^\top-\E( \vX_i\vX_i^\top) \big)\big|_\infty \ge x \big)\\\label{eqn:Y_i,jk_large-dev}
&&\qquad\qquad \le C_1 \nu_{2q}^q {p \varpi_{q,A}(B_n) M_{X,q}^{q} \over B_n^{q}  x^q} + C_2 p^2 \exp\left( -C_3 {B_nx^2 \over \nu_4^2 N_X^2} \right),
\end{eqnarray}
and
\begin{eqnarray}
\nonumber
&&\Prob\big( \sup_{t\in (0,1)}\big|\sum_{i=1}^n w(t,t_i) \big(\vX_i\vX_i^\top-\E( \vX_i\vX_i^\top) \big)\big|_\infty \ge x \big)\\ \label{eqn:Y_i,jk_large-dev_unif}
&&\qquad\qquad\le C_1 \nu_{2q}^q {p \varpi_{q,A}(n) M_{X,q}^{q} \over B_n^q x^q} + C_2 p^2\exp\left( -C_3 {B_n^2x^2 \over n\nu_4^2 N_X^2} \right).
\end{eqnarray}
\end{lem}

\begin{proof}
For $1\le j,k\le p$, let $Y_{i,jk} = X_{ij} X_{ik}$. We now check the conditions in Lemma \ref{lem:nagaev} for $(Y_{i,jk})_{1\le i\le n}$. Denote $Y_{i,jk,\{m\}} = X_{ij,\{m\}} X_{ik,\{m\}}$. Then the uniform functional dependence measure of $(Y_{i,jk})_i$ is
\begin{eqnarray*}
\theta_{m,q,jk}^Y &=&\sup_i \|Y_{i,jk} - Y_{i,jk,\{m\}}\|_q\\
&=& \sup_i \| X_{ij} X_{ik} - X_{ij,\{m\}} X_{ik,\{m\}} \|_q \\
&\le& \sup_i \|X_{ij} (X_{ik} - X_{ik,\{m\}} ) \|_q + \sup_i \|X_{ik,\{m\}} (X_{ij} - X_{ij,\{m\}} ) \|_q.
%&\le& \sup_i \| X_{ij}\|_{2q} \| X_{ik} - X_{ik,\{m\}} \|_{2q} + \sup_i \|X_{ik,\{m\}}\|_{2q} \| X_{ij} - X_{ij,\{m\}} \|_{2q} \\
%&\le& \sup_i \| X_{ij}\|_{2q} \theta_{m,2q,k} + \sup_i \| X_{ik}\|_{2q} \theta_{m,2q,j}.
\end{eqnarray*}
Thus the DAN of the process $Y_{\cdot, jk}$ satisfies that
\begin{align*}
\|Y_{\cdot, jk}\|_{q,A} %&=& \sup_{m \ge 0} (m+1)^A \theta_{m,q,jk}^Y \\
%&\le& \sup_i \| X_{ij}\|_{2q} \;  \sup_{m \ge 0} (m+1)^A \theta_{m,2q,k} + \sup_i \| X_{ik}\|_{2q} \;  \sup_{m \ge 0} (m+1)^A \theta_{m,2q,j} \\
\le  \sup_i \| X_{ij}\|_{2q} \; \|X_{\cdot,k}\|_{2q,A} + \sup_i \| X_{ik}\|_{2q} \; \|X_{\cdot,j}\|_{2q,A}\le \nu_q(\|X_{\cdot,k}\|_{2q,A}+\|X_{\cdot,j}\|_{2q,A}).
\end{align*}
The result follows immediately from Lemma \ref{lem:nagaev} and the Bonferroni inequality. 
\end{proof}

\begin{lem}\label{lemma:covpredev}
We adopt the notation in Lemma \ref{lemma:Y_i,jk_large-dev}. Suppose Assumptions \ref{assumption:dep}, \ref{assumption:smoothness} and \ref{assumption:kernel}  hold with $\iota=0$. Recall $B_n = n b$, where $b\to 0$ and $B_n/\sqrt n\to \infty$ as $n\to\infty$. Then there exists a constant $C$ independent of $n$ and $p$ such that $\hat\Sigma(t)$ in (\ref{eqn:samplecov}) satisfies that for any $t\in [c,1-c]$,
\begin{align}\label{eqn:cov_dev}
|\hat\Sigma(t)-\Sigma(t)|_\infty &= O_\P \left( b ^2+M_{X,q}\nu_{2q}B_n^{-1}(p\varpi_{q,A}(B_n))^{1/q}+ \nu_4N_X (\log{p}/B_n)^{1/2} \right).
\end{align}
Furthermore, 
\begin{align}
\label{eqn:cov_dev_unif}
{\sup_{t \in [c,1-c]} |\hat\Sigma(t)-\Sigma(t)|_\infty}= O_\P\left( b ^2+M_{X,q}\nu_{2q}B_n^{-1}(p\varpi_{q,A}(n))^{1/q}+\nu_4N_XB_n^{-1} [n\log p]^{1/2} \right).
\end{align}
\end{lem}

\begin{proof}
First we have
\begin{align*}
\E\hat\sigma_{jk}(t)-\sigma_{jk}(t) = \sum_{i=1}^n w(t,t_i) [\sigma_{jk}(t_i)-\sigma_{jk}(t)].
\end{align*}
Approximating the discrete summation with integral, we obtain for all $1\le j,k\le p$,
\begin{align*}
\sup_{t\in [b,1-b]}\rBR{\E\hat\sigma_{jk}(t)-\sigma_{jk}(t)-\int_{-1}^{1} K(u)[\sigma_{jk}(u b +t)-\sigma_{jk}(t)]du}=O\left({ B_n^{-1}}\right).
\end{align*}
By Assumption \ref{assumption:smoothness}, we have 
\begin{align*}
\sigma_{jk}(ub+t)-\sigma_{jk}(t)&=u b \sigma'_{jk}(t)+\frac{1}{2}u^2 b ^2\sigma''_{jk}(t)+o( b ^2u^2).
\end{align*}
Thus we have $\sup_{t\in [c,1-c]}|\E\hat\sigma(t)-\sigma(t)|_\infty=O\left({B_n^{-1}}+ b ^2\right)$,  in view of Assumption \ref{assumption:kernel}.  
By Lemma \ref{lemma:Y_i,jk_large-dev}, we have
 \begin{align*}
\sup_{t\in (0,1)} \P\left(\left|\hat\Sigma(t)-\E\hat\Sigma(t)\right|_\infty \ge x\right) %&\le \Prob \left(\left| \sum_{i \in I_t} a_i  Y_i \right|_\infty \ge c_t^{-1/2} x \right) \\\no number
&\le C_1 p \nu_q^q {M_{X,q}^q \varpi_{q,A}(B_n) \over B_n^{q} x^q} + C_2p^2\exp\rbr{-C_3{B_nx^2 \over N_X^2}}.
\end{align*}
Denote $u= C_4\big(M_{X,q}\nu_{2q}B_n^{-1}(p\varpi_{q,A}(B_n))^{1/q}+ \nu_4N_X (\log{p}/B_n)^{1/2}\big)$ for a large enough constant $C_4$, then for any $t\in (0,1)$,
\begin{align}\nonumber
\left|\hat\Sigma(t)-\E\hat\Sigma(t)\right|_\infty=O_\P(u).
\end{align}
Thus (\ref{eqn:cov_dev}) is proved. The result (\ref{eqn:cov_dev_unif}) can be obtained similarly.
\end{proof}

\subsection{Proof of main results}
\begin{proof}[Proof of Proposition \ref{prop:precmx_smooth}]
Given (\ref{eqn:cov_dev}) and (\ref{eqn:cov_dev_unif}), the proof of (\ref{eqn:pre_dev}) is standard. (See, e.g. \cite[Theorem 6]{MR2847973}). For $\lambda^\circ$ and $\lambda^*$ given in Proposition \ref{prop:precmx_smooth}, by Lemma \ref{lemma:covpredev}, we have that respectively, 
\begin{align}\label{eqn:lambdao}
\lambda^\circ&\ge \sup_{t}\E\big(\kappa_p|\hat \Sigma(t)-\Sigma(t)|_\infty\big),\\
\label{eqn:lambda*}
\lambda^\diamond&\ge \E\big(\kappa_p\sup_{t}|\hat \Sigma(t)-\Sigma(t)|_\infty\big).
\end{align}
Then note that for any $t\in[0,1]$, for any $\lambda>0$,
\begin{align*}
&|\hat\Omega_\lambda(t)-\Omega(t)|_\infty \le |\Omega(t)|_{L_1}|\Sigma(t)\hat\Omega_\lambda(t)-\Id_p|_\infty\\
&\qquad\le |\Omega(t)|_{L_1}  \big[|\hat\Sigma(t)\hat\Omega_\lambda(t)-\Id_p|_\infty+|(\Sigma(t)-\hat\Sigma(t))\Omega(t)|_\infty+|\hat\Omega_\lambda(t)-\Omega(t)|_{L^1}|\hat\Sigma(t)-\Sigma(t)|_\infty\big]
\end{align*}
where by construction, we have $|\hat\Sigma(t)\hat\Omega_\lambda(t)-\Id_p|_\infty\le \lambda$ and $|\hat\Omega_\lambda(t)-\Omega(t)|_{L^1}\le 2\kappa_p$. Consequently,
\begin{align}\label{eqn:MT1236p}
|\hat\Omega_\lambda(t)-\Omega(t)|_\infty\le \kappa_p\big(\lambda+3\kappa_p|\hat\Sigma(t)-\Sigma(t)|_\infty\big).
\end{align}
Then (\ref{eqn:pre_dev}) and (\ref{eqn:pre_dev_unif}) follow from (\ref{eqn:lambdao}) to (\ref{eqn:MT1236p}).
\end{proof}

\begin{proof}[Proof of Proposition \ref{prop:partial_support_recovery}]
Theorem \ref{prop:partial_support_recovery} is an immediate result of (\ref{eqn:pre_dev_unif}).
\end{proof}

\begin{proof}[Proof of Theorem \ref{thm:jumppointsestimation}]
Denote $r_j, 1\le j\le \iota$ as the time point(s) of the time of jump ordered decreasingly in the sense of the infinite norm of covariance matrices, i.e., $|\Delta(r_1)|_\infty\ge |\Delta(r_2)|_\infty\ge \ldots\ge |\Delta(r_\iota)|_\infty\ge |\Delta(s)|_\infty$ for $s\in (0,1)\cap \{r_1,\ldots, r_\iota\}^c$. (Temporal order is applied if there is a tie.)  Let ${\calT}_{h}(j)=[r_j-h,r_j+h)$. For $h=o(1)$, as a result of Assumption \ref{assumption:jump}, ${\calT}_{h}(j)\cap {\calT}_{h}(i)=\emptyset$ if $i\neq j$ for $n$ sufficiently large. That is to say, each time point $s\in(0,1)$ is in the neighborhood of at most one change point. 

For any $s\in [t^{(j)},t^{(j+1)})$, $j=0,1,\ldots, \iota$, denote $\mathbb{D}(s)=\E[D(s)]$ and 
\begin{equation}
\label{eqn:D(s)_expectation-nonstationary}
\D^\diamond(s) = \left\{
\begin{array}{cc}
(h-s+t^{(j)}) \Delta(t^{(j)}), & t^{(j)} \le s <t^{(j)} +h \\
0, & t^{(j)} +h \le s < t^{(j+1)}-h \\
(h+s-r) \Delta(t^{(j+1)}), & t^{(j+1)}-h \le s \le t^{(j+1)}.
\end{array}
\right. 
\end{equation}
Then, for $s \in \cup_{1\le j\le \iota}[t^{(j)} +h, < t^{(j+1)}-h)$, by (\ref{eqn:L-cond-1-ture-covariance-matrix}), we have
\begin{equation*}
|\Sigma(s +t) - \Sigma(s) |_\infty \le {L t }, \qquad \forall |t|\le h,
\end{equation*}
we can easily verify that
\begin{equation}\label{eqn:diff_stat_nonstat}
\sup_{s \in [0,1]} |\D(s) - \D^\diamond(s)|_\infty \le Lh^2.
\end{equation}
Note that $|\D^\diamond(s)|_\infty$ is maximized at $s = r_1$ and $|\D^\diamond(r_1)|_\infty = h |\Delta(r_1)|_\infty$.  By the triangle inequalities, we have that for some positive constant $C$, for any $s\in [0,1]$,
\begin{eqnarray}\nonumber
|\D(r_1)|_\infty - |\D(s)|_\infty &\ge& h c_2 - |\D(r_1) - \D^\diamond(r_1)|_\infty - |\D^\diamond(s)|_\infty - |\D(s) - \D^\diamond(s)|_\infty \\\nonumber
&\ge& h c_2 - |\D^\diamond(s)|_\infty -2Lh^2 \\\label{eqn:D_lower}
&\ge&  c_2 (|s-r_1|\wedge h) - 2 L h^2.
\end{eqnarray}
On the other hand, since $|D(r_1)|_\infty\le|D(\hat s_1)|_\infty$, we have
\begin{eqnarray}\nonumber
|\D(r_1)|_\infty-|\D(\hat s_1)|_\infty&\le |D(r_1)|_\infty-|D(\hat s_1)|_\infty+ |\D(r_1)-D(r_1)|_\infty+|\D(\hat s_1)-D(\hat s_1)|_\infty\\\label{eqn:D_upper}
&\le |\D(r_1)-D(r_1)|_\infty+|\D(\hat s_1)-D(\hat s_1)|_\infty.
\end{eqnarray}
Denote the event ${\calA}:= \{\sup_{s\in [h,1-h]} |D(s) - \D(s)|_\infty \le  h_\diamond^2\}$ and let ${\bf Y_i}=(Y_{i,jk})_{1\le j,k\le p}$, $Y_{i,jk} = X_{ij} X_{ik} -\sigma_{i,jk}$. Note that
\begin{equation}
\label{eqn:D<=Y}
|D_{jk}(s) - \D_{jk}(s)| = {1\over n} \left| \sum_{i=1}^{hn} Y_{n_s+1-i,jk} -  \sum_{i=1}^{hn} Y_{n_s+i,jk} \right|.
\end{equation}
By Lemma \ref{lemma:Y_i,jk_large-dev}, we have for any $x > 0$,
\begin{align}
\label{eqn:large-deviation-infty}
\Prob\left(\sup_{s\in [h,1-h]} |D(s) - \D(s)|_\infty \ge  x\right)
%\nonumber
%\le& 2\left[\P\left(\max_{1\le l\le n-kn}|\sum_{i=1}^{l+kn} Y_i|_\infty \ge nx/2\right)+\P\left(\max_{1\le l\le n}|\sum_{i=1}^{l}  Y_i|_\infty \ge nx/2\right)\right]\\
\le  C_1 {p \varpi_{q,A}(n) M_{X,q}^q\nu_{2q}^q \over n^{q} x^q} + C_2 p^2 \exp\left(-C_3 {nx^2\over N_X^{2}}\right).
\end{align}
It follows that 
\[
|\D(r_1)|_\infty-|\D(\hat s_1)|_\infty=O_\P\big(h^{-1}J_{q,A}(n,p)+N_Xh^{-1}(n^{-1}\log (p))^{1/2}\big).
\]
Taking $h=h_\diamond$, we have
\[
|\hat s_1-r_1|=O_\P(h_\diamond^2).
\]
Furthermore we have
\[
\P({\calA})\ge 1-C_1\big({p \varpi_{q,A}(n) M_{X,q}^q \nu_{2q}^q \over n^{q} c_2^q}\big)^{1/3}-C_2p^2\exp\big(-C_3({n\log^2(p)\over N_X^2})^{1/3}\big).
\]

Let ${\calA}_k:=\{\max_{1\le j\le k}|\hat s_j-r_j|\le c_2^{-1}2(L+1)h_\diamond^2\}$ for some $1\le k\le \iota$. Assume ${\calA}_k\subset {\calA}$. Under ${\calA}_k$ we have that $[r_j-h_\diamond,r_j+h_\diamond)\subset \hat{{\calT}}_{2h_\diamond}(j)=:[\hat s_j-2h_\diamond,\hat s_j+2h_\diamond)$ for $1\le j\le k$ and $r_{k+1}\notin\cup_{1\le j\le k}\hat{{\calT}}_{2h_\diamond}(j)$ as a consequence of Assumption \ref{assumption:jump}. According to (\ref{eqn:D_lower}) and (\ref{eqn:D_upper}), we have if ${\calA}$ is true, $|\hat s_{k+1}-r_{k+1}|\le c_2^{-1}2(L+1) h_\diamond^2$, which implies ${\calA}_{k+1}\subset {\calA}$. 
The result (\ref{eqn:jumppoint}) follows from deduction. 

Suppose ${\calA}$ holds. By the choice of $\nu$, as a consequence of (\ref{eqn:diff_stat_nonstat}) and (\ref{eqn:large-deviation-infty}), and that $\nu\ll h_\diamond$, we have that 
\[
\sup_{s\in [0,1]}|D(s)-\D^\diamond(s)|_\infty\le \nu.
\]
 As a result, 
 \[
 \min_{1\le j\le \iota} |D(r_j)|_\infty\ge c_2 h_\diamond-\nu\ge \nu,
 \]
 i.e., $\hat \iota\ge \iota$. On the other hand, since $\cup_{1\le j\le \iota}\hat{\calT}_{2h_\diamond}(j)$ is excluded from the searching region for $s_{\iota+1}$, we have
 \[
 \sup_{s\in \big(\cup_{1\le j\le \iota}\hat{\calT}_{2h_\diamond}(j)\big)^c} |D(s)|_\infty\le \nu.
 \]
  In other words, $\{\hat \iota=\iota\}\subset {\calA}$. Thus (\ref{eqn:jumpnumber}) is proved. 

\end{proof}

\begin{proof}[Proof of Theorem \ref{thm:support_recovery}]
We adopt the notations in the proof of Theorem \ref{thm:jumppointsestimation} and assume that $\calE$ holds. Similar as in Lemma \ref{lemma:covpredev}, we have that by Lemma \ref{lemma:Y_i,jk_large-dev}, for any $t\in (0,1)$,
\begin{align}\nonumber
\left|\hat\Sigma(t)-\E\hat\Sigma(t)\right|_\infty=O_\P(u),
\end{align}
where $u= C_4\big(M_{X,q}\nu_{2q}B_n^{-1}(p\varpi_{q,A}(B_n))^{1/q}+ \nu_4N_X (\log{p}/B_n)^{1/2}\big)$ for a large enough constant $C_4$. 

Since under $\calE$, ${\calT}_{b}(j)\subset \hat{\calT}_{b+h_\diamond^2}(j)$. For $t\in \big(\cup_{1\le j\le \iota}\hat{\calT}_{b+h_\diamond^2}(j)\big)^c\cap [b,1-b]$, we have that for all $1\le j,k\le p$, 
\begin{align*}
\rBR{\E\hat\sigma_{jk}(t)-\sigma_{jk}(t)}
&=\int_{-1}^{1} K(u)[\sigma_{jk}(u b +t)-\sigma_{jk}(t)]du+O\left({ B_n^{-1}}\right)\\
&= b\sigma'_{jk}(t)\int_{-1}^{1} uK(u) du+ \big(\frac{1}{2} b ^2\sigma''_{jk}(t)+o(b^2)\big)\int_{-1}^{1} u^2K(u) du+O\left({ B_n^{-1}}\right)\\
&=O(b^2+B_n^{-1}).
\end{align*}
On the other hand, for $t\in \cup_{1\le j\le \iota}\big(\hat{\calT}_{b+h_\diamond^2}(j)\cap {\calT}_{h_\diamond^2}^c(j)\big)\cup [0,b]\cup[1-b,1]$, due to reflection, we no longer have that differentiability. As a result of the Lipschitz continuity, we get
\begin{align*}
\rBR{\E\hat\sigma_{jk}(t)-\sigma_{jk}(t)}
=\int_{-1}^{1} K(u)[\sigma_{jk}(u b +t)-\sigma_{jk}(t)]du+O\left({ B_n^{-1}}\right)=O(b+B_n^{-1}).
\end{align*}
The result (\ref{eqn:cov_boundary_unif}) follows by the choices of $b$. The rest of the proof are similar as in that of Proposition \ref{prop:precmx_smooth} and Theorem \ref{prop:partial_support_recovery}.
\end{proof}

\bibliographystyle{plain}
\bibliography{tvnetwork}
\end{document}